\documentclass[a4paper,twoside,reqno]{amsart}

\usepackage[T1]{fontenc}
\usepackage[english]{babel}
\usepackage{microtype}
\usepackage[left=2.8cm,right=2.8cm]{geometry}
\usepackage[T1]{fontenc}
\usepackage{lmodern}

\numberwithin{equation}{section}

\usepackage{amsmath,amssymb,mathrsfs,mathtools,latexsym} 

\newcommand{\F}{\mathscr{F}}
\newcommand{\G}{\mathscr{G}}
\newcommand{\eps}{\varepsilon}
\renewcommand{\epsilon}{\eps}
\renewcommand{\phi}{\varphi}
\newcommand{\N}{\mathbb{N}}

\newcommand{\R}{\mathbb{R}}
\newcommand{\RN}{\R^N}
\newcommand{\Snmu}{\mathbb{S}^{N-1}}
\newcommand{\minusz}{\setminus\{0\}}
\newcommand{\LN}{\mathscr{L}^N}
\newcommand{\de}{\,\mathrm{d}}

\newcommand{\weak}{\rightharpoonup}

\newcommand{\h}{\mathscr{H}}
\newcommand{\D}{\mathfrak{D}}
\usepackage{amsthm} 
\usepackage{aliascnt} 

\theoremstyle{plain}
\newtheorem{theorem}{Theorem}[section]

\newaliascnt{proposition}{theorem}
\newtheorem{proposition}[proposition]{Proposition}
\aliascntresetthe{proposition}

\newaliascnt{corollary}{theorem}

\aliascntresetthe{corollary}

\newaliascnt{lemma}{theorem}
\newtheorem{lemma}[lemma]{Lemma}
\aliascntresetthe{lemma}

\newaliascnt{definition}{theorem}
\theoremstyle{definition}

\aliascntresetthe{definition}

\newaliascnt{remark}{theorem}
\newtheorem{remark}[remark]{Remark}
\aliascntresetthe{remark}

\usepackage{amsrefs}

\usepackage{comment}
\usepackage{xcolor}

\usepackage[colorlinks=true, pdfstartview=FitV,linkcolor=purple,citecolor=blue,urlcolor=blue]{hyperref}
\mathtoolsset{showonlyrefs}

\title[Second-order asymptotics of fractional Gagliardo seminorms]{Second-order asymptotics of fractional Gagliardo seminorms  as $s\to1^-$ and convergence of the associated  gradient flows}

\author[A. Kubin]{Andrea Kubin}
\author[V. Pagliari]{Valerio Pagliari}
\author[A. Tribuzio]{Antonio Tribuzio}
\AtEndDocument{
	\bigskip
	\footnotesize{
		\noindent
		Andrea Kubin.
		
		\noindent
		Jyv\"askyl\"an Yliopisto, Matematiikan ja Tilastotieteen Laitos,
		
		\noindent
	 Jyv\"askyl\"a, Finland.
		
		\noindent
		E-mail:				
		\href{mailto:andrea.a.kubin@jyu.fi}{\tt andrea.a.kubin@jyu.fi}.
		
		\smallskip
		
		\noindent
		Valerio Pagliari.
		
		\noindent
		Institute of Analysis and Scientific Computing,
		TU Wien,
		
		\noindent
		Wiedner Hauptstra\ss e 8-10, 1040 Vienna, Austria.
		
		\noindent
		E-mail:	\href{mailto:valerio.pagliari@tuwien.ac.at}{\tt valerio.pagliari@tuwien.ac.at}.
		
		\smallskip
		
		\noindent
		Antonio Tribuzio.
		
		\noindent
		Institute for Applied Mathematics,
		University of Bonn,
		
		\noindent
		Endenicher Allee 60, 53115 Bonn, Germany.
		
		\noindent
		E-mail:	\href{mailto:tribuzio@iam.uni-bonn.de}{\tt tribuzio@iam.uni-bonn.de}.
	}
}

\begin{document}
	\maketitle
			\begin{abstract}
	We study the second-order asymptotic expansion of the $s$-fractional Gagliardo seminorm as $ s\to1^-$ in terms of a higher order nonlocal functional.
	We prove a Mosco-convergence result for the energy functionals and that the $L^2$-gradient flows of the energies converge 
	to the $L^2$-gradient flows of the Mosco-limit.
	\end{abstract}

	\tableofcontents
	
	\section{Introduction}

In recent years, the interest in nonlocal problems has grown significantly in the mathematical community.
This is due to the fact that, by their nature, nonlocal energies require less regularity to be well-defined and hence have a broader applicability than their local equivalent.
In these terms, \emph{convolution-type} functionals are by now considered as a more general counterpart of elastic-type energies appearing, for instance, in mechanics, micromagnetism and related fields.
A challenging question is the study of the asymptotic behaviour of these nonlocal energies as the \emph{interaction horizon} between two points goes to zero.

The prototypical nonlocal energy of convolution-type is the squared \emph{$s$-Gagliardo seminorm}
\begin{equation}\label{eq:funct-intro}
|u|_{H^s(\Omega)}^2:=\int_{\Omega}\int_{\Omega} \frac{|u(y) - u(x)|^2}{|y-x|^{N+2s}} \de y \de x
\end{equation}
where $N\in\N\setminus\{0\}$, $s\in(0,1)$, $u\in L^2(\Omega)$ and $\Omega\subset\RN$ is an open, bounded and Lipschitz regular set.
This has been proved to converge as $s\to1^-$ (after a suitable rescaling) to a multiple of the squared Sobolev seminorm. 
The first result in this direction is given in the seminal work \cite{bourgain2001another} where it is proved
$$
\lim_{s\to1^-}(1-s)|u|_{H^s(\Omega)}^2= \frac{\omega_N}{2} \|\nabla u\|_{L^2(\Omega)}^2
$$
in the sense of pointwise convergence, where $\omega_N$ is the volume of the $N$-dimensional unit ball.
Such a result is indeed proved for every exponent $p\in(1,\infty)$ and it is complemented by a treatment of general nonlocal functionals of convolution-type under suitable decay conditions on the kernels.
The analogous result in the case of linear growth-conditions, \emph{i.e.}\ $p=1$, was studied in \cite{Dav02}, where it is shown that the pointwise limit is exactly (a multiple of) the total variation norm of $\nabla u$.
These results have been proved to hold in the sense of $\Gamma$-convergence in \cite{Pon04prim} and generalized to the case of arbitrary open sets in \cite{LS11}.

In this paper, we are interested in the analysis of the second order expansion, as $s\to1^-$, of the rescaled $s$-Gagliardo seminorm
$$
G_s(u):=(1-s)|u|_{H^s(\RN)}^2,
$$
restricted to functions supported in an open, bounded, Lipschitz set $\Omega\subset\RN$.
To be precise, we prove that the \emph{rate functionals}
$$
\G_s^1(u):= \frac{G_1(u)-G_s(u)}{1-s}, \quad \text{where } G_1(u):=\frac{\omega_N}{2}\|\nabla u\|_{L^2(\RN)}^2,
$$
$\Gamma$-converge as $s\to1^-$ with respect to the strong $L^2$-topology to the nonlocal limit energy
$$
\G^1(u) := \int_{\R^N}\int_{\R^N}\frac{|\nabla u(x)\cdot h|^2\chi_{B(0,1)}(h)-|u(x+h)-u(x)|^2}{|h|^{N+2}}\de h\de x.
$$
Our limit result (cf.\ Theorem \ref{thm:main}), supported by the $L^2$-equicoercivity of the $\G_s^1$ energies, upgrades to a \emph{Mosco-limit}.
We complement this analysis by characterizing the domain of $\G^1$ (cf.\ Lemma \ref{lem:dom}), exploiting its Fourier representation, which consists of the functions $u\in H^1_0(\Omega)$ whose Fourier transform $\hat u$ complies with
$$
\int_{\RN} |\xi|^2\log(1+|\xi|^2)|\hat u(\xi)|^2\de\xi<+\infty.
$$

Our results conclude the asymptotic analysis of the $s$-Gagliardo seminorm up to the second order in the singular limit $s\to1^-$.
To the best of the authors' knowledge, our limit result introduces the \emph{new} quadratic, nonlocal energy $\G^1$, which is close to other $L^1$-based analogs appearing in the second-order study of nonlocal perimeters \cite{cesaroni2020second,KS2023}.
As it is apparent from the characterization of its domain, we can interpret $\G^1$ as an $L^2$-based energy of differentiability order "logarithmically larger" than one.
This can also be na{\"i}vely explained by testing out the energy on H{\"o}lder regular functions since
the singularity of both terms in $\G^1$ (as $|h|\to0$) cancels in the difference if $u$ is sufficiently differentiable.
Indeed, taking $u\in C^{1,\alpha}_c(\Omega)$, by Taylor's expansion and the Fundamental Theorem of Calculus, it holds that $u(x+h)=u(x)+\nabla u(x)\cdot h + O(|h|^{1+\alpha})$ hence 
$$
\G^1(u) = \int_{\R^N}\int_{B(0,1)}\frac{(\nabla u(x)\cdot h)O(|h|^{1+\alpha})}{|h|^{N+2}} \de h-\int_{\R^N}\int_{B(0,1)^c}\frac{|u(x+h)-u(x)|^2}{|h|^{N+2}}\de h\de x.
$$
which is finite for any $\alpha>0$.

The second step in our analysis is the study of the $L^2$-gradient flows associated to $\G_s^1$.
As one may expect from the static asymptotic analysis, the \emph{parabolic-type} flows governed by the "gradient" of $\G_s^1$ converge as $ s\to1^-$ to the gradient flow of $\G^1$  (see Theorem \ref{convheat1ordsto0}).
In this context, thanks to the \emph{$\lambda$-convexity} of the functionals involved, the correct notion of gradient of the energies $\G_s^1$ and $\G^1$ is that of \emph{first variation}.
For regular test functions $\phi \in C^\infty_c(\Omega)$, the first variation of $\G_s^1$ is
\begin{equation*}
\frac{ d }{d t} \G_s^1(u+t\phi)|_{t=0}=\omega_N\Big\langle u,\frac{(-\Delta)-\sigma_{N,s}(-\Delta)^s}{1-s}\phi\Big\rangle,
\end{equation*}
where $ (-\Delta)^s$ is the $s$-fractional Laplace operator and $\sigma_{N,s}>0$ is a scaling factor converging to $1$ as $s\to1^-$.
At least formally, this converges to the first variation of the $\Gamma$-limit $\G^1$ which introduces a \emph{new} operator defined by
\begin{multline*}
\mathfrak{L} \phi(x):= -2N \omega_N \phi(x)+ 4\int_{B(0,1)^c} \frac{\phi(x+h)}{\vert h \vert^{N+2}}\de h \\
- 2\int_{B(0,1)} \frac{2\phi(x)-\phi(x-h)-\phi(x+h)+ h \cdot\nabla^2 \phi(x)h }{\vert h \vert^{N+2}} \de h,
\end{multline*}
for every $\varphi\in C^\infty_c(\Omega)$.
This stability result is obtained as an application of the approach (in general Hilbert spaces) developed in \cite{crismale2023variational}, where it is proved that uniform convexity assumptions (\emph{i.e.}\ $\lambda$-convexity and $\lambda$-positivity) are sufficient conditions to guarantee that $\Gamma$-convergence commutes with the gradient flow structure, exploiting the minimizing-moments and energy dissipation approaches \cite{ambrosio2005gradient,CG2012,sandier2004gamma}.

The literature on local limits of fractional-type functionals, in both the critical limits $s\to0$ and $s\to1$, is by now very vast.
Among the many results, we mention the pivotal paper \cite{ambrosio2011gamma} in which the asymptotic behavior of (relative) \emph{fractional perimeters} is studied in terms of $\Gamma$-convergence (see also \cite{de2022core} for analogous results in the supercritical case $s\to1^+$ and \cite{kubin2024characterization} in which sets with fractional perimeters are characterized via fractional heat flows).
In this context, also results on the second order asympotics can be found in \cite{cesaroni2020second} for the classical fractional perimeter and in \cite{KS2023} for more general convolution energies related to the study of thin ferromagnetic films
(see also \cite{CNP20} for energies associated with $L^1$ kerneles with finite second moment).
The first result regarding the limit as $s \to 0^+$ is \cite{maz2002bourgain}, in which the authors show that the squared $s$-Gagliardo seminorm (rescaled by $s$) pointwise converge to (a multiple of) the squared $L^2$-norm (see also \cite{dipierro2013asymptotics,carbotti2022asymptotics} for similar results in the context of $s$-fractional perimeters).
This asymptotic analysis has been developed also in terms of $\Gamma$-convergence in \cite{crismale2023variational} for the standard $s$-Gagliardo seminorm and in \cite{de20190} for fractional perimeters both in the first and in the second order.
In this latter case, these studies provided two new nonlocal functionals, referred to as \emph{$0$-fractional perimeter} in \cite{de20190} and its corresponding $L^2$-based version in \cite[Theorem 1.4]{crismale2023variational}.
Many variants of the problem have also been considered, for instance incorporating anisotropies \cite{ludwig2014anisotropic}, or having different interaction potentials \cite{de2023parabolic,AMRT2009,MRT2016,BP19}.
For other related nonlocal-to-local results see also \cite{AAB22,MQ15,SM19,DDFG2024,CKS2023} and the references therein.

The paper is structured as follows: after recalling some known results in the literature in Section \ref{sec:subcrit}, in Section \ref{sec:main} we study the Mosco-convergence of the
functionals $\G^1_s$ as $s\to1^-$ to $\G^1$. Then, in Section \ref{convgradflow}, we prove
the convergence of the gradient flows of $\G_s^1$ to the gradient flows of $\G^1$.

\medskip

	\noindent
	{\bf Notation of the paper:}
	For $N\in \N\minusz$, our analysis is settled in the $N$-dimensional space $\RN$.
	We denote the Euclidean inner product between $a,b\in\RN$ as $a\cdot b$ and the corresponding norm as $|a|$.
	We denote by $B(x,r)$ the open ball in $\RN$
	of center~$x$ and radius~$r$.
	We write $B(x,r)^c$ for the complement of $B(x,r)$, while 
	for the topological boundary of $B(0,1)$ we use the symbol $\Snmu$ or $\partial B(0,1)$.
	We denote by $\LN$ the $N$-dimensional Lebesgue measure, and for $ k \in \{0,\dots, N-1\}$ we denote $\mathcal{H}^k$ the $k$-dimensional Hausdorff measure.
	We set $\omega_N\coloneqq \LN(B(0,1))$, so that
	the $(N-1)$-dimensional Hausdorff measure of $\Snmu$ equals $N\omega_N$. With $\Omega $ we denote a open bounded subset of $\R^N$ with Lipschitz continuous boundary.
	Throughout the paper, we write
	$C(*,\cdots,*)$ to indicate a  generic positive constant
	that depends only on $*,\cdots,*$ and
	that may change from line to line.

\section{The subcritical rate of convergence of the Gagliardo seminorms}\label{sec:subcrit}

In this section we analyze the asymptotic behaviour of the scaled $s$-Gagliardo seminorms as $s\to1^-$.
Let $s\in (0,1)$ we define the scaled $s$-Gagliardo seminorm of a measurable function $u : \R^N \rightarrow \R$ as
\begin{equation}\label{eq:gagliardo}
	G_s(u)\coloneqq
	(1-s) \int_{\RN} \int_{\RN} \frac{| u(y)-u(x) |^2}{| x-y |^{N+2s}}\de y \de x.
\end{equation}
For simplicity, we will consider compactly supported functions but the same analysis can be performed by using local convergences.
Let $\Omega$ be a open bounded subset of $\R^N$ with Lipschitz continuous boundary.
In what follows, for every $v\in L^2(\Omega)$ we denote without relabeling the $L^2(\R^N)$ function which is extended to zero outside $\Omega$.
	
\subsection{Preliminary (known) results and first-order limits}
	
The asymptotic analysis of the family $\{G_{s_n}\}_{n \in \N}$ as $s_n\to 0^+$ and as $s_n\to1^-$ in the sense of $\Gamma$-convergence has already been carried out in the literature.
In this respect, we recall the following results.

\begin{theorem}[Theorem 1.2 in \cite{crismale2023variational}]
Let $\{s_n\}_{n \in \N}$ be a sequence such that $s_n \to 0^+$ as $ n \to + \infty$. Then, the following hold.
\begin{itemize}
	\item[(i)] (weak compactness)
	Let $\{u_n\}_{n \in \N}\subset L^2(\Omega)$ satisfy
	\[
	s_nG_{s_n}(u_n) \leq M
	\]
	for some $M\geq 0$ independent of $n$.
	Then, there exists $u \in L^2(\Omega)$ such that, up to subsequences, $u_n $ converge to $u$ in the weak topology of $L^2(\Omega)$.
	\item[(ii)] (liminf inequality) For every $u\in L^2(\Omega)$ and for every $\{u_n\}_{n \in \N}\subset L^2(\Omega)$ converging to $u$ in the weak $L^2(\Omega)$ topology, it holds
$$
\frac{N\omega_N}{2}\|u\|_{L^2(\Omega)}^2 \le \liminf_{n\to+\infty} s_nG_{s_n}(u_n).
$$
	\item[(iii)] (limsup inequality) For every $u\in L^2(\Omega)$ there exists a sequence $\{u_n\}_{n \in \N}\subset L^2(\Omega)$ converging to $u$ in the weak $L^2(\Omega)$ topology such that
$$
\frac{N\omega_N}{2}\|u\|_{L^2(\Omega)}^2 = \lim_{n\to+\infty} s_nG_{s_n}(u_n).
$$
\end{itemize}
\end{theorem}

Similar pointwise convergence results have been obtained in the setting of fractional perimeters, i.e. for energies of linear growth restricted to characteristic functions, (see \cite{dipierro2013asymptotics}) or in the setting of Gaussian fractional perimeter, in which the convolution kernel is given by a Gaussian distribution, (see \cite{carbotti2022asymptotics}).

\begin{theorem}[Theorem 2.1 in \cite{crismale2023variational}]
Let $\{s_n\}_{n \in \N}$ be a sequence such that $s_n \to 1^-$ as $ n \to + \infty$.
Then, the following hold.
\begin{enumerate}
\item (compactness)
	Let $\{u_n\}_{n \in \N}\subset L^2(\Omega)$ satisfy
	\[
	\|u_n \|_{L^2(\Omega)} + G_{s_n}(u_n) \leq M
	\]
	for some $M\geq 0$ independent of $n$.
	Then, there exists $u \in H_{0}^1(\Omega)$ such that, up to subsequences, $u_n \to u$ strongly in $L^2(\Omega)$.
	\item ($\Gamma$-convergence)
	The sequence $\{G_{s_n}\}_{n \in \N}$ $\Gamma$-converges with respect to the strong $L^2(\Omega)$ topology to the functional $G_1$ defined as
	\begin{equation}\label{eq:G1}
		G_1(u)\coloneqq
		\begin{cases}
			\displaystyle{
			\frac{\omega_N}{2} \int_{\Omega} \vert \nabla u(x) \vert^2 \de x} & \text{if } u \in H_0^1(\Omega) \, , \\
			+\infty & \text{otherwise in } L^2(\Omega).
		\end{cases}
	\end{equation}
\end{enumerate}
\end{theorem}
	
For similar compactness and $\Gamma$-convergence results see also \cite{Pon04prim}.
For a pointwise converge result in a general setting we refer \cite{ludwig2014anisotropic}.
An analogous $\Gamma$-convergence result in the setting of the fractional perimeter has been studied in \cite{ambrosio2011gamma}.

For $u \in L^2(\Omega)$, we now introduce the rate functionals $\G^0_s: L^2(\Omega)\to[0,+\infty)$ defined as
\begin{equation*}
	\G^0_s(u) \coloneqq \frac{sG_s(u)-\frac{N\omega_N}{2}\|u\|^2_{L^2(\RN)}}{s}
\end{equation*}
and $\G^1_s: L^2(\Omega)\to[0,+\infty]$ defined as
\begin{equation}\label{G_s^1}
	\G^1_s(u) \coloneqq
	\frac{G_1(u)-G_s(u)}{1-s}.
\end{equation}
	
The $\Gamma$-convergence of the former has been established already.
	
\begin{theorem}[Theorem 1.4 in \cite{crismale2023variational}]
Let $\{s_n\}_{n \in \N}$ be a sequence such that $s_n \to 0^+$ as $ n \to + \infty$.
Then, the following hold.
\begin{enumerate}
	\item (compactness)
	Let $\{u_n\}_{n \in \N}\subset L^2(\Omega$) satisfy
		\[
		\|u_n\|_{L^2(\Omega)}+\G^0_{s_n}(u_n) \leq M
		\]
	for some $M\geq 0$ independent of $n$.
	Then, there exists $u \in \mathcal{H}^0(\Omega)$ such that, up to subsequences, $u_n \to u$ in $L^2(\Omega)$, where the set $\mathcal{H}^0(\Omega)\subset L^2(\RN)$ is defined as
		$$
		\mathcal{H}^0(\Omega) \coloneqq \Big\{v\in L^2(\Omega): \int_{\RN}\int_{\RN}\frac{| v(y)- v(x)|^2}{|y-x|^N}\de y\de x<+\infty \Big\}.
		$$
	\item ($\Gamma$-convergence)
	The sequence $\{\G^0_{s_n}\}_{n \in \N}$ $\Gamma$-converge to the functional $\G^0$ defined as
		$$
		\G^0 (u) \coloneqq \begin{cases}
			\displaystyle
			\frac{1}{2}\iint\limits_{|y-x|\le1}\frac{|v(y)-v(x)|^2}{|y-x|^N}\de y\de x-\iint\limits_{|y-x|>1}\frac{u(x)u(y)}{|y-x|^N}\de y\de x & u\in \mathcal{H}^0(\Omega) \\
			+\infty & \text{otherwise in } L^2(\Omega)
		\end{cases}
		$$
	with respect to the $L^2(\Omega)$ topology.
\end{enumerate}
\end{theorem}
We refer to  \cite{de20190} for the analogous compactness and $\Gamma$-convergence results in the setting of the fractional perimeter.
	
\section{The main limit result}
\label{sec:main}
	
We now focus on the analysis of $\{\G_{s}^1\}_{s\in(0,1)}$ as $s\to1^-$.
We will prove that the energies $\G_s^1$ \emph{Mosco-converge} to a limit functional $\G^1$.
We remind that Mosco-convergence (see \cite{mosco1967approximation}), \emph{i.e.}\ when weak $\Gamma$-liminf and strong $\Gamma$-limsup coincide, implies $\Gamma$-convergence.

Before stating the main convergence result, we first define and discuss the limit functional $\G^1$, since its formulation is quite involved.
	
\subsection{The candidate Mosco-limit}
We introduce the candidate $\Gamma$-limit $\G^1: L^2(\Omega) \rightarrow (-\infty, \, + \infty]$ defined as
\begin{equation}\label{eq:G1-Glim-new}
\begin{split}
\G^1(u) &:= \int_{\R^N}\int_{\R^N}\frac{|\nabla u(x)\cdot h|^2\chi_{B(0,1)}(h)-|u(x+h)-u(x)|^2}{|h|^{N+2}}\de h\de x,
\end{split}
\end{equation}
if $ u \in \mathfrak{D}(\G^1) \subset L^2(\Omega)$ and $\G^1(u)=+\infty$ if $ u \in L^2(\Omega) \setminus \mathfrak{D}(\G^1)$, where $\mathfrak{D}(\G^1)$ is the set of functions $u \in H_0^1(\Omega)$ such that
\begin{equation}\label{eq:domaindef}
\frac{ \int_{0}^{1} \vert \nabla u(x+th) \cdot h \vert^2 \de t-\vert u(x+h)-u(x)\vert^2}{\vert h \vert^{N+2}} \chi_{B(0,1)}(h) \in L^1(\R^N \times \R^N).
\end{equation}

\begin{remark}[alternative form of $\G^1$]\label{rmk:alt}
The limit functional $\G^1$, as defined in \eqref{eq:G1-Glim-new}, can be written in an equivalent form as
\begin{equation}\label{eq:G1-Glim}
\begin{split}
\G^1(u) &= \int_{\R^N}\int_{B(0,1)}\frac{\int_0^1|\nabla u(x+th)\cdot h|^2\de t-|u(x+h)-u(x)|^2}{|h|^{N+2}}\de h\de x \\
&\qquad - \int_{\R^N}\int_{B(0,1)^c}\frac{|u(x+h)-u(x)|^2}{|h|^{N+2}}\de h\de x,
\end{split}
\end{equation}
by exploiting Fubini's Theorem and a change of variable in the gradient term.
Indeed, by \eqref{eq:domaindef} we can switch the orders of the integrals and using the change of variable $x'=x+th$ we get
$$
\int_{R^N}\int_{B(0,1)}\int_0^1\frac{|\nabla u(x+th)\cdot h|^2}{|h|^{N+2}}\de t\de h\de x = \int_{R^N}\int_{B(0,1)}\frac{|\nabla u(x')\cdot h|^2}{|h|^{N+2}}\de h\de x'.
$$
This, less compact, formulation has the advantage to separate the singular integral (first term) from the absolutely convergent one (second term).
With this alternative formulation \eqref{eq:G1-Glim} it is evident why \eqref{eq:domaindef} defines the domain of $\G^1$.

Another advantage of this is also the fact that the function in \eqref{eq:domaindef} is non-negative by a straightforward application of the Fundamental Theorem of Calculus and Jensen's inequality.
We will often prefer formulation \eqref{eq:G1-Glim} in the sequel, especially in computations.
\end{remark}

Even though the characterization of the domain $\mathfrak{D}(\G^1)$ is quite involved, we can prove that this set is dense in $L^2(\Omega)$ thanks to the following result.

\begin{lemma}\label{lemma03112023pom1}
Let $u \in C_c^{2}(\Omega)$ then
$$
\frac{ \int_{0}^{1} \vert \nabla u(x+th) \cdot h \vert^2 \de t-\vert u(x+h)-u(x)\vert^2}{\vert h \vert^{N+2}} \chi_{B(0,1)}(h) \in L^{1}(\R^N \times \R^N);
$$
i.e., $C_c^2(\Omega) \subset \mathfrak{D}(\G^1)$.
\end{lemma}
\begin{proof}
We preliminarily observe that for every $h\in B(0,1)$
\begin{multline*}
	\frac{\int_{0}^{1} \vert \nabla u(x+th) \cdot h \vert^2 \de t-\vert u(x+h)-u(x)\vert^2}{\vert h \vert^{N+2}} \\
	=\int_{0}^1 \frac{\big|\big(u(x+h)-u(x)- \nabla u(x+th)\cdot h\big)\big(u(x+h)-u(x)+\nabla u(x+th)\cdot h\big)\big|}{\vert h \vert^{N+2}} \de t,
\end{multline*}
where we also used that the left-hand side is nonnegative.
By the regularity of $u$, we have that for all $t\in [0,1]$
\begin{equation}\label{eq:taylor}
\vert u(x+h)-u(x)-\nabla u(th+x)\cdot h \vert \leq R \vert h \vert^2
\end{equation}
where 
$$
R:= C(N)\sup \Big\{ \Big\vert \frac{\partial^2 u}{\partial x_i \, \partial x_j}(x) \Big\vert \, : x \in \Omega, \, i,j = 1, \cdots,N \Big\}.
$$ 
An application of the triangle inequality yields
\begin{equation}\label{07122023pom1}
	\begin{split}
	\int_{\R^N} &\int_{B(0,1)}\int_{0}^1 \frac{\vert \big(u(x+h)-u(x)- \nabla u(th+x)\cdot h\big)\big(u(x+h)-u(x)+ \nabla u(x+th)\cdot h\big) \vert}{\vert h \vert^{N+2}} \, \de h \de t \de x  \\
	&\le \int_{\R^N} \int_{B(0,1)} \int_{0}^1\frac{\vert u(x+h)-u(x)-\nabla u(x+th)\cdot h \vert \vert u(h+x)-u(x) \vert }{\vert h \vert^{N+2}} \, \de h \de t  \de x \\
	& \qquad + \int_{\R^N} \int_{B(0,1)} \int_{0}^1\frac{\vert u(h+x)-u(x)-\nabla u(x+th)\cdot h \vert \vert \nabla u(x+th) \cdot h \vert }{\vert h \vert^{N+2}} \, \de h \de t \de x.
	\end{split}
\end{equation}
It is now sufficient to control the two terms at the right-hand side of \eqref{07122023pom1}.
Regarding the first, from \eqref{eq:taylor} we infer
\begin{equation*}
	\begin{split}
	\int_{\R^N} &\int_{B(0,1)}\int_{0}^1\frac{\vert u(h+x)-u(x)-\nabla u(x+th)\cdot h \vert \vert u(h+x)-u(x) \vert }{\vert h \vert^{N+2}}  \de h \de t  \de x\\
	&\leq  R \int_{\R^N} \, dx\,\int_{B(0,1)}\frac{\vert  u (h+x)-u(x) \vert }{\vert h \vert^{N}} \, dh \\
	& \leq R \int_{\R^N} \int_{B(0,1)} \int_{0}^{1} \frac{\vert \nabla u(x+sh)\cdot h \vert }{\vert h \vert^{N}} \de x \de h \de s\\
	&= R \int_{\R^N} \vert \nabla u (\xi) \vert \de \xi \int_{B(0,1)} \frac{1}{\vert h \vert^{N-1}} \de h= R N \omega_N \int_{\R^N} \vert \nabla u (x) \vert \de x\, .
	\end{split}
\end{equation*}
Arguing similarly as above we also get
\begin{equation*}
	\begin{split}
	\int_{\R^N} &\int_{B(0,1)} \int_{0}^1\frac{\vert u(h+x)-u(x)-\nabla u(th+x)\cdot h \vert \vert \nabla u(th+x) h \vert }{\vert h \vert^{N+2}} \de h \de t  \de x\\
	&\leq R \int_{\R^N} \int_{B(0,1)}\frac{\vert \nabla u (\xi) \vert }{\vert h \vert^{N-1}} dh \de \xi = R N \omega_N \int_{\R^N} \vert \nabla u (\xi) \vert \de \xi \,.
	\end{split}
\end{equation*}
Thus, by the two formulas above and \eqref{07122023pom1} we obtain the claim.		
\end{proof}

\subsection{Characterization of the domain via Fourier transform}\label{sec:domain-F}

As we already commented on, the definition of $\D(\G^1)$ is not very transparent, even though, as Lemma \ref{lemma03112023pom1} shows, it is well-defined in $L^2(\Omega)$, in particular it is dense in the strong $L^2$ topology.
Further interesting comments on the set $\D(\G^1)$ are apparent using the Fourier transform.

Nonlocal, $L^2$-based quadratic functionals can be rewritten in a local form thanks to the Fourier transform see \cite{davoli2023sharp ,di2012hitchhikers}. 
To fix notation, for every $u\in C^\infty_c(\R^N)$ we write
	$$
	\mathcal{F}u(\xi):=(2\pi)^{-\frac{N}{2}}\int_{\R^N} e^{-i x\cdot\xi} u(x) dx,
	$$
and this extends (as a linear operator) to the whole $L^2(\R^N)$.
When there is no risk of confusion we will write $\hat u(\xi)=\mathcal{F}u(\xi)$.

Noticing that $\mathcal{F}(u(\cdot+h)-u)(\xi)=(e^{-i h\cdot\xi}-1)\hat u(\xi)$ and that $\mathcal{F}(\nabla u)(\xi)=i\xi\hat u(\xi)$, the Fourier transform (in $x$ for fixed $h$ and $t$) of the term in \eqref{eq:domaindef} reads $|h|^{-N-2}\big(|\xi\cdot h|^2-2+2\cos(\xi\cdot h)\big)|\hat u(\xi)|^2$.
Denoting as
\begin{equation}\label{eq:mult}
m(\xi):=\int_{B(0,1)}\frac{|\xi\cdot h|^2-2+2\cos(\xi\cdot h)}{|h|^{N+2}}\de h,
\end{equation}
by Plancherel's and Fubini's Theorem we get that for every $u\in H^1_0(\Omega)$, $u\in\D(\G^1)$ if and only if
	\begin{equation}\label{eq:dom-four}
	\int_{R^N} m(\xi)|\hat u(\xi)|^2 \de \xi<+\infty.
	\end{equation}
By the change of variable $\eta=|\xi|h$, we can rewrite $m$ as
	\begin{equation}\label{eq:multiplier}
	m(\xi)=|\xi|^2\int_{B(0,|\xi|)}\frac{|\bar{\xi}\cdot \eta|^2-2+2\cos(\bar \xi\cdot\eta)}{|\eta|^{N+2}}\de \eta
	\end{equation}
where $ \bar{\xi}= \frac{ \xi}{\vert \xi \vert}$.
From this we can see that \eqref{eq:dom-four} is a differentiability condition of order strictly higher than one.
By a careful analysis of the decay of $m(\xi)$, we can give a complete characterization of the domain $\D(\G^1)$.

\begin{lemma}\label{lem:dom}
Let $u\in H^1_0(\Omega)$.
Then \eqref{eq:domaindef} holds if and only if
\begin{equation}\label{eq:decay}
\int_{\R^N}|\xi|^2\log(|\xi|^2+1)|\hat u(\xi)|^2 \de\xi<\infty.
\end{equation}
In particular $\D(\G^1)=\{u\in H^1_0(\Omega) : \eqref{eq:decay} \text{ holds true }\}$ and $ (\D(\G^1), \|\cdot \|_{\D(\G^1)})$ with $  \|u \|_{\D(\G^1)} := \| \nabla u \|_{L^2(\Omega)}+ (\int_{\R^N}|\xi|^2\log(|\xi|^2+1)|\hat u(\xi)|^2 \de\xi )^{\frac{1}{2}}$ is a Banach space.
\end{lemma}
\begin{proof}
The result follows by providing upper and lower bounds for the multiplier defined in \eqref{eq:multiplier}.

We start by giving an upper bound.
Consider first the case $|\xi|\le2$.
Then, since $\cos(t)\le 1-\frac{t^2}{2}+\frac{t^4}{24}$ for every $t\in\R$, we get
\begin{align*}
m(\xi) \le |\xi|^2 \frac{1}{12}\int_{B(0,|\xi|)}\frac{|\bar{\xi}\cdot\eta|^4}{|\eta|^{N+2}}\de\eta \le \frac{N\omega_N}{24}|\xi|^4.
\end{align*}
For $|\xi|>2$ we thus have
\begin{align*}
m(\xi) &\le \frac{1}{6}N\omega_N|\xi|^2 + |\xi|^2\int_{B(0,|\xi|)\setminus B(0,2)}\frac{|\bar{\xi}\cdot\eta|^2}{|\eta|^{N+2}}\de\eta \\
&\le \frac{1}{6}N\omega_N|\xi|^2 + N\omega_N |\xi|^2\int_1^{|\xi|}\frac{1}{t}\de t = N\omega_N|\xi|^2 \Big( \frac{1}{6}+\log(|\xi|)\Big).
\end{align*}
Hence for some constant $C_2$ depending on $N$ it holds
$$
m(\xi) \le C_2|\xi|^2\big(1+\log(|\xi|^2+1)\big).
$$

We now pass to the lower bound.
We preliminarily notice that, straightforward computations yield that $t^2-2+2\cos(t)\ge\frac{5}{9}t^2$, for every $t\in\R$ with $|t|\ge3$, while for $|t|< 3$ we have $t^2-2+2\cos(t)\ge\frac{1}{48}t^4$.
We also denote $E_\xi:=\{\eta\in\R^N\setminus\{0\}:\sqrt{2}\bar{\xi}\cdot\eta\ge|\eta|\}$.
Hence, for $|\xi|<3$ we have
\begin{align*}
m(\xi) &\ge |\xi|^2 \frac{1}{48} \int_{B(0,|\xi|)}\frac{|\bar{\xi}\cdot\eta|^4}{|\eta|^{N+2}}\de\eta \\
&\ge |\xi|^2 \frac{1}{192}\int_{B(0,|\xi|)\cap E_\xi}\frac{1}{|\eta|^{N-2}}\de\eta \ge \frac{N\omega_N}{768}|\xi|^4.
\end{align*}
If instead $|\xi|\ge3$ we get
\begin{align*}
m(\xi) &\ge \frac{1}{768}N\omega_N|\xi|^4+|\xi|^2 \frac{5}{9} \int_{B(0,|\xi|)\setminus B(0,3)}\frac{|\bar{\xi}\cdot\eta|^2}{|\eta|^{N+2}}\de\eta \\
&\ge \frac{1}{768}N\omega_N|\xi|^4+|\xi|^2 \frac{5}{18}\int_{\big(B(0,|\xi|)\setminus B(0,3)\big)\cap E_\xi}\frac{1}{|\eta|^{N}}\de\eta \\
&\ge \frac{1}{768}N\omega_N|\xi|^4+\frac{5}{36}N\omega_N|\xi|^2(\log(|\xi|)-\log(3)).
\end{align*}
This implies that there exists a constant $C_1$ depending on $N$ such that
$$
m(\xi) \ge C_1 |\xi|^2 \big(\log(|\xi|^2+1)-1\big)
$$
and the result is proven.
\end{proof}

\subsection{Second-order $\Gamma$-convergence of Gagliardo seminorms}
	
We are now in position to state our result, that is the second-order asymptotics, via $\Gamma$-convergence, of the $s$-Gagliardo seminorms.
As said before, we will prove a stronger result, namely the Mosco-convergence of the functionals $\G^1_s$ as $ s\to1^-$.

\begin{theorem}\label{thm:main}
Let $\{s_n\}_{n \in \mathbb{N}} \subset (0,1)$ be such that $s_n \to1^-$ as $ n \to + \infty$,
and let $\G^1_s$ and $\G^1$ be defined as in \eqref{G_s^1} and \eqref{eq:G1-Glim-new} respectively.
Then, the following Mosco-convergence results hold true.
\begin{enumerate}
	\item (compactness) Let $\{u_n\}_{n\in\N}\subset L^2(\Omega)$ be such that
	\begin{equation*}\label{eq:comp-2-ord}
	\sup_{n \in \mathbb{N}} \G^1_{s_n}(u_n) <+\infty
	\quad \text{and} \quad
	\sup_{n \in \mathbb{N}} \|u_n\|_{L^2(\Omega)}<+\infty.
	\end{equation*} 
	Then, up to subsequences, $u_n \rightarrow u$ in the weak topology of $H^1_0(\Omega)$ for some $u\in \mathfrak{D}(\G^1)$.
	\item (weak liminf inequality) For every $u \in L^2(\Omega)$ and for every $ \{u_n\}_{n \in \N} \subset L^2(\Omega)$ with $ u_n \weak u$ in the weak topology of $L^2(\Omega)$, it holds
	\begin{equation}\label{liminfineq}
	\G^1(u)\leq \liminf_{n \rightarrow + \infty} \G^1_{s_n}(u_n).
	\end{equation}
	\item (strong limsup inequality) For every $u \in L^2(\Omega)$ there exists $ \{u_n\}_{n \in \N} \subset L^2(\Omega)$ with $ u_n \rightarrow u$ in $L^2(\Omega)$ such that
	\begin{equation}\label{limsupineq}
	\G^1(u)\geq  \limsup_{n \rightarrow + \infty} \G^1_{s_n}(u_n).
	\end{equation}
\end{enumerate}
\end{theorem}
	
The next subsections are devoted to the proof of Theorem \ref{thm:main}.
	
An analogous result in the case of $\Gamma$-convergence of fractional perimeters has been obtained in \cite{cesaroni2020second}.
Our strategy adapts the one in that work to the current setting.

\subsubsection{Coercivity and lower semicontinuity of $ \G_s^1$} 

In the next proposition we prove the lower semicontinuity of the functionals $\G_s^1$ for fixed $s\in(0,1)$.
We also provide a compactness result (in the weak $H^1$ topology) 
for sequences with bounded $L^2$-norm and bounded $\G_s^1$-energy.

\begin{proposition}\label{proplsc+com+sfix}
Let $s \in (0,1)$  be fixed and let $\G^1_s$ be defined as in \eqref{G_s^1}.
Let $ \{u_n\}_{n \in \N}\subset L^2(\Omega)$ such that
\begin{equation}\label{eq:unif-bound}
\| u_n \|_{L^2(\Omega)} \leq m \quad \text{and} \quad \G^1_s(u_n) \leq M
\end{equation}
for some $M$, $m>0$ and for every $n \in \N$, then $\|\nabla u_n\|_{L^2(\Omega)} \leq M'$ for every $n \in \N$, for some positive constant $M'$ depending on $M$, $s$ and $N$.

In particular, the sequence $ \{u_n\}_{n \in \N}$ converges in the weak $H^1$ topology, up to a subsequence, to a function $u\in H^1_0(\Omega)$.

Moreover, the functional $\G_s^1$ is lower semicontinous with respect to the weak $L^2$ topology.  
\end{proposition}
\begin{proof}
Notice first that for every $u \in H^1_0(\Omega)$ with $ \| u \|_{L^2(\Omega)} \leq m$ by  the interpolation inequality, see \cite{leoni2023first, van2023fractional},  we have
\begin{equation}\label{30102023pom1}
	\begin{split}
\int_{\R^N} \int_{\R^N} \frac{\vert u(x)-u(y)\vert^2}{\vert x-y \vert^{N+2s}}\, \de x \de y &\leq C\| u \|_{L^2(\Omega)}^{2(1-s)} \| \nabla u \|_{L^2(\Omega)}^{2 s} \\
	&\leq Cm^{2(1-s)} \| \nabla u \|_{L^2(\Omega)}^{2 s}=Cm^{2(1-s)}N^s G_1(u)^{s}.
\end{split}
\end{equation}
Let $ \{u_n\}_{n\in \N}$ be as in \eqref{eq:unif-bound}, by definition of $\G_s^1$ and by \eqref{30102023pom1} we have
\begin{equation*}
	(1-s)M \geq G_1(u_n)-G_s(u_n) \geq G_1(u_n)-(1-s)Cm^{2(1-s)}N^s G_1(u_n)^{s}.
\end{equation*}
From this we conclude that necessarily $\|\nabla u_n\|_{L^2(\Omega)} = \big(\frac{2}{\omega_N}G_1(u_n)\big)^\frac{1}{2} \leq M'$, where $M'$ is a constant which depends on $m$, $M$, $s$ and $N$.
As a consequence, by the Rellich-Kondrachov theorem, we have that $u_n$ is compact in the weak $H^1$ topology and its cluster points are in $H^1_0(\Omega)$ by \eqref{eq:unif-bound}.

Now, consider a sequence $\{v_n\}_{n\in\N}\subset L^2(\Omega)$ such that $v_n \weak v$ in $L^2(\Omega)$ and that $\G_s^1(v_n) \leq M$.
Since, in particular $\|v_n\|_{L^2(\Omega)}$ is bounded, by the previous argument we get that $\|\nabla v_n\|_{L^2(\Omega)} \leq M'$ where $M'$ depending on $M$, $s$, and $N$.
Again by the Rellich-Kondrachov theorem we have and that $v_n \rightarrow u$ in $L^2(\Omega)$.
Notice now that, by H{\"o}lder's inequality, we get
\begin{align*}
&|G_s(v_n)-G_s(v)| \\
& \quad \le \int_{\R^N}\int_{\R^N}\frac{|(v_n(x)+v(x))-(v_n(y)+v(y))||(v_n(x)-v(x))-(v_n(y)-v(y))|}{|x-y|^{N+2s}}\de x \de y \\
& \quad \le G_s(v_n+v)^\frac{1}{2} G_s(v_n-v)^\frac{1}{2}.
\end{align*}
Using again the interpolation inequality as in the first line of \eqref{30102023pom1} we obtain
$$
|G_s(v_n)-G_s(v)| \le \tilde C \|v-v_n\|_{L^2(\Omega)}^{2(1-s)}
$$
for some $\tilde C$ depending on $s$, $N$ and $M'$.
This implies the continuity of $G_s$ with respect to the strong $L^2$ topology.
This, along with the (well known) lower semicontinuity of $G_1$ with respect the weak $H^1$ convergence yields the conclusion.
\end{proof}

A crucial remark in our analysis is that $\G^1_s$ can be rewritten as a sum of two terms resembling the formulation \eqref{eq:G1-Glim} of the limit functional $\G^1$; one term continuously convergent (as $s\to 1^-$) and the other monotone increasing in $s$.
This on the one hand will facilitate the proof of the $\Gamma$-convergence result, essentially reducing it to proving a pointwise limit, and on the other will allow us to transfer the coerciveness of $\G^1_{\bar s}$ (for some $\bar{s}$ fixed) to the equicoerciveness of the sequence $\G^1_s$.

\begin{lemma}\label{lem:Gs-dec}
	Let $u \in H_0^1(\Omega)$ then for every $s \in (0,1)$, we have
	\begin{equation}\label{eq:Gs-dec}
		\begin{split}
			\G^1_s(u) &=  -\frac{N\omega_N}{s}\int_{\R^N} \vert u(x) \vert^2 \, \de x+ 2 \int_{\R^N}\int_{B(0,1)^c} \frac{u(x+h)u(x)}{\vert h \vert^{N+2s}} \de h\de x \\
			& \qquad +\int_{\R^N} \int_{B(0,1)} \frac{\int_{0}^{1} \vert \nabla u(th+x) \cdot h \vert^2 \de t-\vert u(x+h)-u(x)\vert^2}{\vert h \vert^{N+2s}} \de h \, \de x.
		\end{split}
	\end{equation}
\end{lemma}

\begin{proof}
	We observe that for all $u \in H_0^1(\Omega)$
	\begin{equation}\label{eq:Gs-dec-p1} 
		\begin{split}
			\int_{\R^N} \int_{\R^N} \frac{\vert u(x+h)-u(x) \vert^2}{\vert h \vert^{N+2s}}\de x\, \de h &= \frac{N\omega_N}{s}\int_{\R^N} \vert u(x) \vert^2 \de x-2\int_{\R^N}\int_{B(0,1)^c} \frac{u(x+h)u(x)}{\vert h \vert^{N+2s}} \de h\de x \\
			& \qquad + \int_{\R^N} \int_{B(0,1)} \frac{\vert u(h+x)-u(x) \vert^2}{\vert h \vert^{N+2s}}\de h\de x.
		\end{split}
	\end{equation}
	By summing and subtracting $\int_0^1|\nabla u(x+th)\cdot h|^2\de t$, we rewrite the last term on the right-hand side of \eqref{eq:Gs-dec-p1} as
	\begin{equation}\label{eq:Gs-dec-p2}
		\begin{split}
			\int_{\R^N} &\int_{B(0,1)} \frac{\vert u(h+x)-u(x) \vert^2}{\vert h \vert^{N+2s}}\de h \de x = \int_{\R^N} \int_{B(0,1)} \int_{0}^1 \frac{\vert \nabla u(x+th) \cdot h \vert^2}{\vert h \vert^{N+2s}} \de t\de h \de x \\
			& \qquad + \int_{\R^N} \int_{B(0,1)} \frac{\vert u(h+x)-u(x) \vert^2 -\int_{0}^{1} \vert \nabla u(x+th) \cdot h \vert^2 \de t}{\vert h \vert^{N+2s}} \de h\de x.
		\end{split}
	\end{equation}
	We now compute the first integral on the right-hand side above using the integration in polar coordinate and the change of variable $\xi=x+th$, denoting $\nu= h / \vert h \vert$;
	\begin{equation}\label{eq:Gs-dec-p3}
		\begin{split}
			\int_{\mathbb{R}^N} \int_{B(0,1)} \int_{0}^{1} \frac{\vert \nabla u(x+th) \cdot \frac{h}{\vert h \vert} \vert^2 }{\vert h\vert^{N+2s-2}} \de t\de h\de x &
			=\int_{\mathbb{R}^N}\int_{B(0,1)}\int_{0}^{1} \frac{\vert \nabla u (\xi) \cdot \frac{h}{\vert h \vert} \vert^2}{\vert h \vert^{N+2s-2}} \,\de t \, \de h \,\de \xi \\
			&= \int_{\mathbb{R}^N} \int_{0}^{1} \rho^{1-2s} \int_{\partial B(0,1)} \vert \nabla u(\xi) \cdot \nu \vert^2 \, \de  \mathcal{H}^{N-1}(\nu)  \de\rho\de\xi \\
			&= \frac{\omega_N}{2(1-s)} \int_{\mathbb{R}^N} \vert \nabla u(\xi) \vert^2 \, \de \xi,
		\end{split}
	\end{equation}
	where in the last step we have used that
	\begin{equation*}
		\begin{split}
			\int_{\partial B(0,1)} \vert \nabla u(\xi) \cdot\nu \vert^2 \,\de \mathcal{H}^{N-1}(\nu) &= \int_{\partial B(0,1)} \vert \nabla u(\xi) \vert^{2} \bigg| \frac{\nabla u(\xi) }{\vert \nabla u(\xi)\vert}\cdot\nu \bigg|^2\, \de \mathcal{H}^{N-1}(\nu) \\
			&=\vert \nabla u(\xi) \vert^{2} \int_{\partial B(0,1)} \vert e_i \cdot \nu \vert^2 \,\de  \mathcal{H}^{N-1}(\nu) \\
			&=  \frac{\vert \nabla u(\xi) \vert^2}{N} \sum_{i=1}^{N} \int_{\partial B(0,1)} \vert e_i \cdot \nu \vert^2 \, \de  \mathcal{H}^{N-1}(\nu)= \vert \nabla u(\xi) \vert^2 \omega_N.
		\end{split}
	\end{equation*}
	Inserting formulas \eqref{eq:Gs-dec-p2} and \eqref{eq:Gs-dec-p3} into \eqref{eq:Gs-dec-p1} we obtain
	\begin{equation*}
		\begin{split}
			\int_{\R^N} \int_{\R^N} & \frac{\vert u(h+x)-u(x) \vert^2}{\vert h \vert^{N+2s}}\de x\de h- \frac{\omega_N}{2(1-s)} \int_{\R^N} \vert \nabla u(x) \vert^2 \de x \\
			& =\frac{N\omega_N}{s}\int_{\R^N} \vert u(x) \vert^2 \de x- \int_{\R^N} \int_{B(0,1)^c} \frac{u(x+h)u(x)}{\vert h \vert^{N+2s}}  \de h \de x \\
			& \qquad +\int_{\R^N} \int_{B(0,1)} \frac{\vert u(h+x)-u(x) \vert^2 -\int_{0}^{1} \vert \nabla u(x+th) \cdot h \vert^2 \de t}{\vert h \vert^{N+2s}} \de h\de x
		\end{split}
	\end{equation*}	
	and the claim is proved.
\end{proof}

As commented above, the first two terms on the right-hand side of \eqref{eq:Gs-dec} converges locally uniformly as $s\to1^-$.
We then focus on the third term.
For the sake of brevity, we introduce the following notation;
\begin{equation}\label{31102023pom1}
\mathcal{J}_s(u):=\int_{\R^N} \int_{B(0,1)} \frac{\int_{0}^{1} \vert \nabla u(x+th) \cdot h \vert^2 \de t-\vert u(x+h)-u(x)\vert^2}{\vert h \vert^{N+2s}} \de h \de x.
\end{equation}

In the next lemma we prove the positivity and monotonicity of $\mathcal{J}_s$ which will be useful to prove the $\Gamma$-convergence, see \cite{braides2002gamma, dal2012introduction}.

\begin{lemma}[positivity and monotonicity]\label{monotonicitylemma}
	For $s\in(0,1)$ let $\mathcal{J}_s$ be defined as in \eqref{31102023pom1}.
	Then, for every $u\in H^1_0(\R^n)$, $\mathcal{J}_s(u)\ge 0$ and the map $s \mapsto \mathcal{J}_s(u)$ is monotone increasing in $s$, \emph{i.e.}\ $\mathcal{J}_{s_1}(u)\le \mathcal{J}_{s_2}(u)$ for $0<s_1\le s_2<1$.
\end{lemma}
\begin{proof}
The claim is a straightforward consequence of the fundamental theorem of calculus, Jensen's inequality and the fact that $s\mapsto|h|^{-N-2s}$ is increasing in $B(0,1)$.
\end{proof}

We end this section by noticing that for all $ s \in (0,1]$ the functionals 
\begin{equation}\label{C0-itaintuu}
	(u,v) \mapsto \int_{\R^N} \int_{B(0,1)^c} \frac{u(x+h)v(x)}{\vert h \vert^{N+2s}}\,\de h\de x
\end{equation} 
are continuous in $L^2(\Omega) \times L^2(\Omega)$ since, by H{\"o}lder's inequality,
\begin{equation}\label{01112023matt1}
	\Big|\int_{\R^N} \int_{B(0,1)^c} \frac{u(x+h)v(x)}{\vert h \vert^{N+2s}}\,\de h\de x \Big| \leq \frac{N\omega_N}{2s}  \| u \|_{L^2(\Omega)} \| v \|_{L^2(\Omega)}.
\end{equation}
We will exploit this property quite often in the sequel.

\subsubsection{Proof of the strong limsup inequality}
We now prove the pointwise convergence $\G^1_s$ which, in particular, yields the strong limsup inequality,
claimed in Theorem \ref{thm:main}(iii).

\begin{proof}[Proof of Theorem \ref{thm:main}(iii)]
	We claim that for every $u \in L^2(\Omega)$
	\begin{equation}\label{prop:plim}
		\lim_{s\to1^-} \G^1_s(u)= \G^1(u).
	\end{equation}
	By Lemma \ref{lem:Gs-dec} it is sufficient to study the pointwise limit of $\mathcal{J}_s$ since the first two terms on the right-hand side of \eqref{eq:Gs-dec} converge pointwise (actually locally uniformly) to the same expression with $s=1$, \emph{e.g.}\ by Lebesgue's dominated convergence theorem. Indeed, since for all $h$, $x \in \R^N$
$$
\int_{0}^{1} \vert \nabla u(x+th) \cdot h \vert^2 \de t-\vert u(h+x)-u(x) \vert^2\geq 0,
$$
and $\mathcal{J}_s$ is increasing in $s$ by Lemma \ref{monotonicitylemma}, by monotone convergence
\begin{equation}\label{eq:plim1}
\lim_{s\to1^-} \mathcal{J}_s(u) = \int_{\R^N} \int_{B(0,1)} \frac{\int_{0}^{1} \vert \nabla u(th+x) \cdot h \vert^2 \de t-\vert u(h+x)-u(x) \vert^2}{\vert h \vert^{N+2}}\de h\de x.
\end{equation}
This (as explained at the beginning of the proof) implies the pointwise limit of $\G^1_s$ which yields the claim by taking $u_n\equiv u$ as a recovery sequence.
\end{proof}

The pointwise limit of the monotone family $\mathcal{J}_s$ directly implies the $\Gamma$-convergence of $\G^1_s$ to $\G^1$; indeed increasing sequences $\Gamma$-converge to their pointwise limit (if this is lower semicontinuous) \cite{dal2012introduction} and by stability under continuous perturbations \cite{dal2012introduction}.

Since we will prove the stronger Mosco-convergence, we need to provide first a compactness result.

\subsubsection{Proof of the compactness}
The next proposition is devoted to the proof of the compactness property claimed in Theorem \ref{thm:main}(i).
To prove it, we follow the strategy exploited in the proof of \cite[Proposition 2.6]{cesaroni2020second}, combining the monotonicity of $\mathcal{J}_s$ (Lemma \ref{monotonicitylemma}) and the coercivity of the energy as $s$ fixed (Proposition \ref{proplsc+com+sfix}).
\begin{proposition}[compactness]\label{compcompcompthm}
Let $s_n \to 1^-$ and let $\G^1_s$ be defined as in \eqref{G_s^1}.
Let $\{ u_n\}_{n \in \N} \subset L^2(\Omega)$ such that
$$
\|u_n\|_{L^2(\Omega)} \leq m \quad \text{and} \quad \G_s^1(u_n) \leq M
$$
for all $n \in \N$, for some $m>0$ and $M\in\R$ independent of $n$.
Then $\|\nabla u_n\|_{L^2(\Omega)}\leq M'$ for every $n \in \N$ for some $M'>0$ depending on $M$, $m$ and $N$, and the sequence $\{u_n\}_{n \in \N}$ converges in the weak $H^1$ topology, up to a subsequence, to a function $u \in \mathfrak{D}(\G^1)$.
\end{proposition}
\begin{proof}
	Without loss of generality, we can assume that $s_n\nearrow1$.
	From Proposition \ref{proplsc+com+sfix} we have that  $u_n \in H_0^1(\Omega)$ for all $ n \in \N$.
	By Lemma \ref{monotonicitylemma} we obtain
	\begin{equation}\label{30102023pom2}
		\begin{split}
			|M| \geq  \G_{s_n}^1(u_n)
			&= \mathcal{J}_{s_n}(u_n)-\int_{\R^N} \int_{B(0,1)^c} \frac{\vert u_n(x+h)-u_n(x) \vert^2}{\vert h \vert^{N+2s_n}}\de h\de x  \\
			&\geq \mathcal{J}_{s_n}(u_n)-2 m^2\frac{N\omega_N}{s_n}
		\end{split}
	\end{equation}
	where in the last inequality we have used that
\begin{multline*}
\int_{\R^N} \int_{B(0,1)^c} \frac{\vert u_n(x+h)-u_n(x) \vert^2}{\vert h \vert^{N+2 s_n}}\de h\de x  \\
=  -2\int_{\R^N} \int_{B(0,1)^c} \frac{u_n(x+h)u_n(x) }{\vert h \vert^{N+2 s_n}} \de h\de x+ \frac{N \omega_N}{s_n} \int_{\R^N} \vert u_n(x) \vert^2 \de x
\end{multline*}
	 and inequality \eqref{01112023matt1}.		
	Fix now $\bar{n} \in \N$ such that $ s_{\bar{n}}>\frac{1}{2}$.
We claim that there exists $M'$, depending on $M,m,N$ but independent of $n$, such that $\|\nabla u_n\|_{L^2(\Omega)} \leq M'$ for every $n\in\N$.
	Indeed, for every $n\geq \bar{n}$, using the monotonicity of $s \mapsto \mathcal{J}_s(u_n)$, see Lemma \ref{monotonicitylemma}, and by \eqref{30102023pom2} we have that
	\begin{equation}\label{09112023pom1}
\begin{split}
		|M| +2m^2\frac{N\omega_N}{s_n} &\geq \mathcal{J}_{s_n}(u_n) \geq \mathcal{J}_{s_{\bar{n}}}(u_n) = \G_{s_{\bar{n}}}^1(u_n)-\int_{\R^N} \int_{B(0,1)^c} \frac{\vert u_n(x+h)-u_n(x)\vert^2}{\vert h \vert^{N+2s_{\bar{n}}}}\de h\de x \\
&\geq  \G_{s_{\bar{n}}}^1(u_n) - 2N\omega_Nm^2.
	\end{split}
\end{equation}
	This implies in particular that $ \G_{s_{\bar{n}}}^1(u_n) \leq \vert M \vert +6m^2\frac{N\omega_N}{s_n}$, and we obtain the claim using Proposition \ref{proplsc+com+sfix}.  
Moreover, by \eqref{09112023pom1} and Fubini's theorem we get
$$
+\infty > \liminf_{n\to+\infty} \mathcal{J}_{s_{\bar n}}(u_n)=\liminf_{n\to+\infty} \int_{B(0,1)}\int_{\R^N}\frac{|\nabla u_n(x)\cdot h|^2-|u_n(x+h)-u_n(x)|^2}{|h|^{N+2s_{\bar n}}}\de x \de h.
$$
By weak $H^1$ convergence of $u_n$ to $u$ (thus strong in $L^2$) and the (sequentially) weak lower semicontinuity of the Sobolev seminorm we further infer that
$$
\int_{B(0,1)}\int_{\R^N}\frac{|\nabla u(x)\cdot h|^2-|u(x+h)-u(x)|^2}{|h|^{N+2s_{\bar n}}}\de x \de h < +\infty.
$$
Since this holds for every $\bar n\in\N$, by \eqref{eq:plim1} we conclude that $u \in \mathfrak{D}(\G^1)$ and the result is proved.
\end{proof}

\subsubsection{Proof of the weak liminf inequality}
We now turn to the proof of the weak liminf inequality in Theorem \ref{thm:main}(ii).
Before proving the desired result, we show it for the family $\{\mathcal{J}_s\}_{s \in (0,1)}$.

\begin{lemma}\label{lemma310102023pom}
Let $\mathcal{J}_s$ be defined as in \eqref{31102023pom1}.
For every $s_n \nearrow 1$ and for all $u_n \weak u$ in the weak $L^2$ topology such that
\begin{equation}\label{eq:bound-J}
\mathcal{J}_{s_n}(u_n) \leq M,
\end{equation}
for some $M>0$ independent of $n$, there holds
		\begin{equation}\label{31102023pom3}
				\liminf_{n \rightarrow +\infty} \mathcal{J}_{s_n}(u_n) \geq \int_{\R^N} \int_{B(0,1)} \frac{\int_{0}^{1} \vert \nabla u(th+x) \cdot h \vert^2 \de t-\vert u(x+h)-u(x) \vert^2}{\vert h \vert^{N+2}} \de h \de x .
		\end{equation}
	\end{lemma}
	\begin{proof}
		Fix $ \bar{s} <1$.
		By the monotonicity property provided by Lemma \ref{monotonicitylemma}, \eqref{eq:bound-J} and \eqref{01112023matt1} we get
$$
\G_{\bar s}^1(u_n) \leq C.
$$
		Thus Proposition \ref{proplsc+com+sfix} yields that $ u_n \rightarrow u $ in the weak $H^1$ topology. Therefore we have
		\begin{equation}
			\begin{split}
				\liminf_{n \rightarrow + \infty} \mathcal{J}_{s_n}(u_n) &\geq \liminf_{n \rightarrow +\infty} \mathcal{J}_{\bar s} (u_n) \\
				&\geq \liminf_{n \rightarrow + \infty} \G_{\bar s}^1(u_n)+\lim_{n \rightarrow + \infty} \int_{\mathbb{R}^{N}} \int_{B(0,1)^{c}} \frac{\vert u_n(x+h-)u_n(x) \vert^2}{\vert h \vert^{N+2\bar{s}}}\de h\de x\\
				&\geq  \G_{\bar s}^1(u)+ \int_{\mathbb{R}^{N}} \int_{B(0,1)^{c}} \frac{\vert u(x+h)-u(x) \vert^2}{\vert h \vert^{N+2\bar{s}}}\de h\de x = \mathcal{J}_{\bar s}(u)
			\end{split}
		\end{equation}
		where we have used for the last inequality the lower semicontinuity proved in Proposition \ref{proplsc+com+sfix}, and the continuity of the map 
		$$ u \mapsto  \int_{\R^N} \int_{B(0,1)^c} \frac{\vert u(x+h)-u(x) \vert^2}{\vert h \vert^{N+2 \bar{s}}}  \de h\de x  $$
		with respect to the strong $L^2$ topology. 
Eventually, taking the limit as $ \bar{s} \rightarrow 1^-$, the result follows by the pointwise convergence of $\mathcal{J}_{\bar{s}}$ \eqref{eq:plim1}.
	\end{proof}

	We have now everything in place to give the proof of Theorem \ref{thm:main}(ii), which then conclude completely the proof of Theorem \ref{thm:main}.

	\begin{proof}[Proof of Theorem \ref{thm:main}(ii)]
		Let $ s_n \nearrow 1$ and let $ u_n \weak u $ in the weak $L^2$ topology. We will prove that 
		\begin{equation*}
			\liminf_{n \rightarrow + \infty}  \G_{s_n}^1(u_n) \geq \G^1(u).
		\end{equation*}
	We can assume that $\liminf_{n \rightarrow + \infty}  \G_{s_n}^1(u_n) < +\infty$, otherwise there is nothing to prove.
By Proposition \ref{compcompcompthm} we have $u_n \rightarrow u$ in the weak $H^1(\Omega)$ topology.
By definition of $ \mathcal{J}_{s_n}$, see formula \eqref{31102023pom1}, we have
		\begin{equation}\label{18022024matt3}
			\liminf_{n \rightarrow + \infty}  \G_{s_n}^1(u_n) \geq \liminf_{n \rightarrow + \infty} \mathcal{J}_{s_n}(u_n)+ \liminf_{n \rightarrow + \infty} -\int_{\mathbb{R}^{N}} \int_{B(0,1)^{c}} \frac{\vert u_n(x+h)-u_n(x) \vert^2}{\vert h \vert^{N+2s_n}}\de h\de x.
		\end{equation}
		Since $u_n \rightarrow u$ in $L^2(\Omega)$, we can assume that $u_n\to u$ almost everywhere, up to a subsequence, and there exists $f \in L^2(\Omega)$ such that $\vert u_n \vert \leq f$.
Hence by Lebesgue's dominated convergence theorem we have
$$
\lim_{n\to+\infty} \int_{\mathbb{R}^{N}} \int_{B(0,1)^{c}} \frac{\vert u_n(x+h)-u_n(x) \vert^2}{\vert h \vert^{N+2s_n}}\de h\de x = \int_{\mathbb{R}^{N}} \int_{B(0,1)^{c}} \frac{\vert u(x+h)-u(x) \vert^2}{\vert h \vert^{N+2}}\de h\de x.
$$
		By Lemma \ref{lemma310102023pom}, from \eqref{18022024matt3} and the formula above we infer that
		\begin{equation*}
			\begin{split}
				\liminf_{n \rightarrow + \infty} \G^1_{s_n}(u) &\geq  -\frac{2}{N \omega_N }\int_{\R^N} \vert u(x) \vert^2 \, \de x+2 \int_{\R^N} \int_{B(0,1)^c} \frac{u(x+h)u(x)}{\vert h \vert^{N+2}} \de h\de x \\
				&  \qquad + \int_{\R^N} \int_{B(0,1)}\frac{\int_{0}^{1} \vert \nabla u(x+th) \cdot h \vert^2 \de t-\vert u(x+h)-u(x) \vert^2}{\vert h \vert^{N+2}}\de h \de x.
			\end{split}
		\end{equation*}	
		This is the desired result, recalling \eqref{eq:G1-Glim}. 
	\end{proof}

\section{Stability of $\G_s^1$ gradient flows under $\Gamma$-convergence}\label{convgradflow}
In this section we investigate the stability of the $\G_s^1$ gradient flows as $s\to1^-$.
After computing the first variation of $\G_s^1$ for all $s\in (0,1)$ and of $\G^1$, we will recall and apply a general result of convergence of gradient flows of $\Gamma$-converging functionals that are uniformly $\lambda$\emph{-convex} and $\lambda$ \emph{-positive} (see \cite[Definition 3.1]{crismale2023variational}).
We will then prove that $\{\G_s^1\}_{s \in (0,1)}$ complies with these two latter properties, and hence we will obtain the desired stability result.

Here and below $\langle\cdot,\cdot\rangle$ denotes the standard scalar product in $L^2(\Omega)$. In this section we denote by $v_t$ the partial time derivative of a function $v$.

\subsection{First variation of $ \G_s^1$ and $\G^1$} 
For the convenience of the reader, we recall the definition and the main properties of the $s$-fractional laplacian in the lines below, as it is part of the first variation of our rate energy $\G^1_s$.

For every $\phi \in C^\infty_c(\R^N)$ we define the $s$-fractional Laplacian of $ \phi $ as
\begin{equation}
(- \Delta)^{s}\phi(x):=C(N,s)\frac{1}{2}\int_{\R^N} \frac{2 \phi(x) -\phi(x+h)-\phi(x-h)}{\vert h \vert^{N+2s}} \de h \quad \text{for every } x \in \R^N,
\end{equation}
where $C(N,s)$ is the multiplicative factor
$$
C(N,s):=\Big(\int_{\R^N}\frac{1-\cos(h_1)}{|h|^{N+2s}}\de h\Big)^{-1},
$$
see \emph{e.g.}\ \cite[formulas (3.1) and (3.2)]{di2012hitchhikers}.
It is well-known (cf.\ \cite[Corollary 4.2]{di2012hitchhikers}) that $C(N,s)$ behaves like $(1-s)$ as $s\to 1^{-}$, in the sense that $\frac{C(N,s)}{1-s}\to \frac{4}{\omega_N}$.
In \cite[Lemma 3.2]{di2012hitchhikers} it is proved that for all $\phi \in C_c^{\infty}(\R^N)$
	\begin{equation}
		(-\Delta)^s \phi \in L^{\infty}(\R^N) \text{ and } (-\Delta)^s \phi (x) = C(N,s)\lim_{r \rightarrow 0^+} \int_{\R^N \setminus B(0,r)} \frac{\phi(x)-\phi(x+h)}{\vert h \vert^{N+2s}} \de h.
	\end{equation}
	It is well known that the $s$-fractional laplacian is related to the squared Gagliardo $s$-seminorm; \emph{i.e.}, the $s$-fractional laplacian is the first variation of the functional 
	$$u \mapsto \frac{C(N,s)}{4(1-s)}G_s(u)=\frac{C(N,s)}{4} \int_{\R^N} \int_{\R^N} \frac{ \vert u(y)-u(x)\vert^2}{\vert y-x \vert^{N+2s}}\de y\de x,$$
with $G_s$ defined as in \eqref{eq:gagliardo}, for a proof of this see \cite[Proposiotion 4.1]{crismale2023variational}. As customary, we extend the fractional laplacian to every $u \in L^2(\Omega)$ such that $G_s(u)< \infty$ (namely $u\in H_0^s(\Omega)$ ) via duality, by defining the $s$-fractional laplacian of $u$ as follows
	\begin{equation*}
		\langle(-\Delta)^su,\phi\rangle:= \langle u,(-\Delta)^s\phi \rangle \quad \text{for all } \phi \in C^{\infty}_{c}(\Omega).
	\end{equation*}
	For all $ u \in H_0^1(\Omega)$ and for all $\phi \in C_c^{\infty}(\Omega) $ we have that the first variation of $u \mapsto \G_s^1(u)$ is 
	\begin{equation}\label{06112023matt1}
		\lim_{t \rightarrow 0}  \frac{ \G_s^1(u+t \phi)- \G_s^1(u)}{t}=\Big\langle u,\frac{\omega_N(-\Delta)-\frac{4(1-s)}{C(N,s)}(-\Delta)^s}{1-s}\phi \Big\rangle .
	\end{equation}
We now define a new operator that we will prove to be the first variation of the $\Gamma$-limit $\G^1$.
Let $\phi \in C^\infty_c(\R^N)$ we define 
	\begin{equation}\label{1-fractionallaplace}
		\begin{split}
			\mathfrak{L} \phi(x):=& -2N \omega_N \phi(x)+ 4\int_{B(0,1)^c} \frac{\phi(x+h)}{\vert h \vert^{N+2}}\de h \\
			& \qquad -2\int_{B(0,1)} \frac{2\phi(x)-\phi(x-h)-\phi(x+h)+ h \cdot\nabla^2 \phi(x)h }{\vert h \vert^{N+2}} \de h.
		\end{split}
	\end{equation}
This new operator is well-defined as results by the following computations.

\begin{lemma}
	Let $ \phi \in C_c^3(\Omega)$ then $ \mathfrak{L} \phi \in C(\Omega)$.
\end{lemma}
\begin{proof}
	We first notice that
	\begin{equation}\label{02112023pom2}
	\sup_{x \in \Omega}\Big| \int_{B(0,1)^c} \frac{\phi(x+h)}{\vert h \vert^{N+2}}\de h\Big| \leq \| \phi \|_{L^{\infty}} \int_{B(0,1)^c} \frac{1}{\vert h \vert^{N+2}} \de h = \frac{N \omega_N}{2} \| \phi \|_{L^{\infty}}.
	\end{equation}
	Using the Taylor expansion of the function $\phi$ we have that for all $ (x,h) \in \Omega_1 \times B(0,1)$
	\begin{equation}\label{02112023pom1}
	\big\vert 2\phi(x)-\phi(x-h)-\phi(x+h)+ h \cdot\nabla^2 \phi(x)h \big\vert \leq R\vert h \vert^3 
	\end{equation}
	where  $\Omega_1:= \{z \in \R^N\, : \, \mathrm{dist}(z, \Omega ) \leq 1\}$ and
	$$
	R:= C(N)\sup \Big\{ \Big\vert \frac{\partial^3 \phi }{\partial x_i \, \partial x_j \partial x_k}(x) \Big\vert \, : x \in \Omega_1, \, i,\,j,\,k = 1, \dots,N \Big\}.
	$$ 
	Therefore by formula \eqref{02112023pom1} we obtain
	\begin{equation}\label{02112023pom3}
	\begin{split}
	\sup_{x \in \Omega}\Big| \int_{B(0,1) } \frac{2\phi(x)-\phi(x-h)-\phi(x+h) + h \cdot\nabla^2 \phi(x)h }{\vert h \vert^{N+2}} & \de h \Big| \\
	& \hspace{-5ex} \leq R \int_{B(0,1)} \frac{1}{\vert h \vert^{N-1}} \de h = R N \omega_N.
	\end{split}
	\end{equation}
	By \eqref{02112023pom2} and \eqref{02112023pom3} we obtain $\mathfrak{L} \phi \in L^{\infty}(\Omega)$.
	The continuity of $ x \mapsto \mathfrak{L} \phi(x)$ is a consequence of Lebesgue's dominated convergence theorem, which can be applied by exploiting \eqref{02112023pom3} and yields sequential continuity.
	\end{proof}

	As for the $s$-fractional laplacian, for every $u \in\mathfrak{D}(\G^1)$ we define $\mathfrak{L}u$ by duality as follows
	\begin{equation}
		\langle\mathfrak{L} u,\phi\rangle:= \langle u,\mathfrak{L} \phi \rangle \quad \text{for all } \phi \in C^{\infty}_{c}(\Omega).
	\end{equation}
	The operator $\mathfrak{L}$ is the first variation of the  functional $\G^1$, as shown below.
	\begin{lemma}\label{varprimaG^1}
		For every $ u \in \mathfrak{D}(\G^1)$ and for every $ \phi \in C_{c}^{\infty}(\Omega)$ we have
		\begin{equation}\label{firstvariationG^1}
			\lim_{t \rightarrow 0} \frac{ \G^1(u+t\phi)-\G^1(u)}{t}= \langle \mathfrak{L}  u, \,  \phi \rangle.
		\end{equation}
	\end{lemma}
	\begin{proof}
	To shorten the computations, we define the following functionals 
	\begin{align*}
		\mathcal{A} \colon L^2(\Omega) \rightarrow (-\infty,+\infty), & \quad \mathcal{A}(u):= -N \omega_N \int_{\R^N} \vert u(x) \vert^2 \, \de x, \\
		\mathcal{B} \colon L^2(\Omega) \rightarrow (-\infty,+\infty) ,& \quad \mathcal{B}(u):= 2 \int_{\R^N}\int_{B(0,1)^c}  \frac{u(x+h)u(x)}{\vert h \vert^{N+2}}\de h\de x.
	\end{align*}
	For all $r \in (0,1)$ we also define the bilinear functionals $ \mathcal{C}_r\colon L^2(\Omega)\times L^2(\Omega) \rightarrow (-\infty,+\infty]$ 
	\begin{equation*}
		\mathcal{C}_r(u,v):=- \int_{S_r}\frac{( u(h+x)-u(x)-\nabla u(x+th)\cdot h )(v(h+x)-v(x)+\nabla v(x+th)\cdot h)}{\vert h \vert^{N+2}}  \de x\de h \de t 
	\end{equation*}
	where $S_r:= \R^N \times \big(B(0,1) \setminus B(0,r)\big)\times(0,1)$.
	By \eqref{eq:G1-Glim}, for all $u \in \mathfrak{D}(\G^1)$ we have that 
	\begin{equation}\label{03112023pom1}
		\G^1(u)= \mathcal{A}(u)+\mathcal{B}(u)+ \lim_{r \rightarrow 0^+} \mathcal{C}_r(u,u).
	\end{equation}
	By formula \eqref{03112023pom1} we have that
	\begin{equation}
	\begin{split}
		\lim_{t \rightarrow 0} \frac{ \G^1(u+t\phi)-\G^1(u)}{t} &= \lim_{t \rightarrow 0} \frac{ \mathcal{A}(u+t\phi)-\mathcal{A}(u)}{t}+\lim_{t \rightarrow 0} \frac{ \mathcal{B}(u+t\phi)-\mathcal{B}(u)}{t}\\
		& \qquad +\lim_{t \rightarrow 0} \lim_{r \rightarrow 0^+}\frac{\mathcal{C}_r(u+t \phi, u+t \phi)-\mathcal{C}_r(u, u)}{t}.
	\end{split}
	\end{equation}
	By standard algebraic computations, we obtain
	\begin{equation*}
		\mathcal{A}(u+t\phi)= \mathcal{A}(u)+t^2 \mathcal{A}(\phi)-2tN \omega_N \int_{\R^N} u(x)\phi(x) \de x
	\end{equation*}
	and
	\begin{equation*}
		\mathcal{B}(u+t \phi)= \mathcal{B}(u)+ t^2 \mathcal{B}(\phi)+4t \int_{\R^N} u(x)\int_{B(0,1)^c} \frac{\phi(x+h) }{\vert h \vert^{N+2}} \de h\de x.
	\end{equation*}
	This yields that
	\begin{equation}\label{02112023sera1}
	\begin{split}
		\lim_{t \rightarrow 0} \frac{ \mathcal{A}(u+t\phi)-\mathcal{A}(u)}{t} &+\lim_{t \rightarrow 0} \frac{ \mathcal{B}(u+t\phi)-\mathcal{B}(u)}{t} \\
		&= \int_{\R^N}u(x)\Big(-2N \omega_N\phi(x)+ 4\int_{B(0,1)^c} \frac{\phi(x+h) }{\vert h \vert^{N+2}} \de h\Big)\de x.
	\end{split}
	\end{equation}
	From the bilinearity of functional, we get that 
	\begin{equation}\label{02112023sera2}
		\mathcal{C}_r(u+t \phi,u+t \phi)= \mathcal{C}_r(u,u)+ t \mathcal{C}_r(u,\phi)+ t \mathcal{C}_r(\phi,u)+ t^2 \mathcal{C}_r(\phi,\phi).
	\end{equation}
	Notice that, by \eqref{eq:domaindef} and Lemma \ref{lemma03112023pom1}, for all $r>0$ there holds $0\leq \mathcal{C}_r(u,u)< +\infty$, $0\le\mathcal{C}_r(\phi,\phi)< +\infty$ and $0\le\mathcal{C}_r(u+t \phi,u+t \phi) < +\infty$.
	Thus, by formula \eqref{02112023sera2} above, we infer that $\mathcal{C}_r(u,\phi)+\mathcal{C}_r(\phi,u)\neq \infty$ for all$r>0$.
	Therefore by \eqref{02112023sera2} we obtain
		\begin{equation}
			\lim_{t \rightarrow 0} \lim_{r \rightarrow 0^+} \frac{\mathcal{C}_r(u+t \phi, u+t \phi)-\mathcal{C}_r(u, u)}{t}= \lim_{r \rightarrow 0^+}\mathcal{C}_r(u,\phi)+  \mathcal{C}_r(\phi,u).
	\end{equation}
	We now claim that 
	\begin{equation}\label{03112023matt9}
	\begin{split}
		\lim_{r \rightarrow 0^+}&\mathcal{C}_r(u,\phi)+  \mathcal{C}_r(\phi,u)\\
		&=	 - 2 \int_{\R^N} u(x)\,\lim_{r \rightarrow 0^+} \int_{B(0,1) \setminus B(0,r) } \frac{2\phi(x)-\phi(x-h)-\phi(x+h)+ h \cdot\nabla^2 \phi(x)h }{\vert h \vert^{N+2}} \de h \de x;
	\end{split}
	\end{equation}
	we show it by expanding the product at the numerator of $\mathcal{C}_r(u,\varphi)$ and $\mathcal{C}_r(\varphi,u)$ (cf.\ definition of $\mathcal{C}_r)$ and rewriting all the terms in a convenient way.
	Regarding the first term of $\mathcal{C}_r(u,\varphi)$, by a change of variable we obtain that
	\begin{equation}\label{03112023matt2}
	\begin{split}
		\int_{\R^N} & \int_{B(0,1)\setminus B(0,r)} \int_{0}^1\frac{ u(h+x) (\phi(x+h)-\phi(x)+\nabla \phi(x+th)\cdot h)}{\vert h \vert^{N+2}} \de h \de t \de x \\
		& \quad = \int_{\R^N} u(x)\int_{B(0,1)\setminus B(0,r)}\int_{0}^1\frac{\phi(x)-\phi(x-h)+\nabla \phi(x+(t-1)h)\cdot h}{\vert h \vert^{N+2}} \de h \de t  \de x.
	\end{split}
	\end{equation}
	The second terms simply reads
	\begin{equation}\label{03112023matt3}
	\int_{\R^N} u(x)\int_{B(0,1)\setminus B(0,r)} \int_0^1\frac{-\phi(x+h)+\phi(x)-\nabla \phi(x+th)\cdot h}{\vert h \vert^{N+2}} \de h \de t \de x
	\end{equation}
	After an integration by parts and a change of variables, the third term reads
	\begin{equation}\label{03112023matt4}
	\begin{split}
		&\int_{\R^N}\int_{B(0,1)\setminus B(0,r)}\int_0^1\frac{-\nabla u(x+th)\cdot h (\phi(x+h)-\phi(x)+\nabla \phi(x+th)\cdot h)}{\vert h \vert^{N+2}} \de h \de t \de x \\
		& \quad =\int_{\R^N} u(x+th)\int_{B(0,1)\setminus B(0,r)} \int_0^1 \frac{\nabla \phi(x+h)\cdot h- \nabla \phi(x)\cdot h+ h \cdot \nabla^2 \phi(x+th) h}{\vert h \vert^{N+2}} \de h \de t  \de x \\
		& \quad = \int_{\R^N} u(x) \int_{B(0,1) \setminus B(0,r)} \int_0^1 \frac{\nabla \phi(x+(1-t)h)\cdot h- \nabla \phi(x -th)\cdot h+ h \cdot \nabla^2 \phi(x) h}{\vert h \vert^{N+2}} \de h \de t \de x .
	\end{split}
	\end{equation}
	Hence, by \eqref{03112023matt2}, \eqref{03112023matt3} and \eqref{03112023matt4} we get
	\begin{equation}\label{eq:c1}
	\begin{split}
		&\mathcal{C}_r(u,\varphi) = \int_{\R^N} u(x) \int_{B(0,1) \setminus B(0,r)} \int_0^1\frac{2\varphi(x)-\varphi(x-h)-\varphi(x+h)+h\cdot \nabla^2\varphi(x) h}{|h|^{N+2}}\de h\de t\de x \\
		& \quad +\int_{\R^N} u(x) \int_{B(0,1) \setminus B(0,r)} \int_0^1\frac{2\nabla\varphi(x+(1-t)h)\cdot h-\nabla\varphi(x+th)\cdot h-\nabla\varphi(x-th)\cdot h}{|h|^{N+2}}\de h\de t\de x.
	\end{split}
	\end{equation}
	Analogous calculations also give that
	\begin{equation}\label{eq:c2}
	\begin{split}
		&\mathcal{C}_r(\varphi,u) = \int_{\R^N} u(x) \int_{B(0,1) \setminus B(0,r)} \int_0^1\frac{2\varphi(x)-\varphi(x-h)-\varphi(x+h)+h\cdot \nabla^2\varphi(x) h}{|h|^{N+2}}\de h\de t\de x \\
		& \quad -\int_{\R^N} u(x) \int_{B(0,1) \setminus B(0,r)} \int_0^1\frac{2\nabla\varphi(x+(1-t)h)\cdot h-\nabla\varphi(x+th)\cdot h-\nabla\varphi(x-th)\cdot h}{|h|^{N+2}}\de h\de t\de x.
	\end{split}
	\end{equation}
	Gathering \eqref{eq:c1} and \eqref{eq:c2} formula \eqref{03112023matt9} is proved.
		Eventually, by \eqref{1-fractionallaplace}, \eqref{02112023sera1}, \eqref{03112023matt9} we have \eqref{firstvariationG^1}.
	\end{proof}

\subsubsection{Fourier symbols of $\mathfrak{L}$}
For completeness, we conclude the discussion about the new operator $\mathfrak{L}$ by computing its Fourier symbols.
This comes as an immediate consequence of
the next lemma in which we compute the Fourier transform of $\mathfrak{L} \phi$ for any $\phi \in C_c^{\infty}(\R^N)$.
\begin{lemma}\label{lem:symb}
	For every $ \phi \in C_c^\infty(\R^N)$ we have that
	\begin{equation}\label{fouririri}
		\mathcal{F}[ \mathfrak{L} \phi](\xi)= \bigg(-2 N \omega_N + 4 \int_{B(0,1)^c} \frac{\cos(\xi\cdot h) }{\vert h\vert^{N+2}}dh +2m(\xi)\bigg) \hat\phi(\xi)
	\end{equation}
where $ \xi \mapsto m(\xi)$ is the function defined in \eqref{eq:mult}.
\end{lemma}
\begin{proof}
	The thesis follows directly by the definition of $\xi \mapsto m(\xi)$ and by the well known properties of the Fourier transform;
	\begin{equation}
	\mathcal{F}[\phi(\cdot+h)](\xi)= e^{-i \xi \cdot h} \hat\phi(\xi),
	\quad
	\mathcal{F}[h \cdot \nabla^2 \phi(\cdot) h](\xi)= - \vert \xi \cdot h \vert^2 \hat\phi(\xi).
	\end{equation}
Eventually, since $\rho\mapsto\sin(\xi\cdot(\rho\sigma))$ is even for every $\xi\in\R^N$ and $\sigma\in\mathbb{S}^{N-1}$, the result is proved.
\end{proof}

We notice that, an alternative proof of the lemma above may come from computing the first variation of Fourier representation of $\G^1$ in $L^2(\R^N;\mathbb{C})$.
Indeed, by the same arguments as in Section \ref{sec:domain-F} we have
$$
\G^1(u)=\int_{\R^N} \Big(m(\xi) -\int_{B(0,1)^c}\frac{2-2\cos(\xi\cdot h)}{|h|^{N+2}}\Big)|\hat u(\xi)|^2 \de \xi,
$$
for every $u\in\D(\G^1)$.
Straightforward computations leads to the same conclusion of Lemma \ref{lem:symb}.

\subsection{Abstract stability results for gradient flows.}
It is well known that, under suitable uniform convexity assumptions, $\Gamma$-convergence commutes with gradient flows, see \cite{brezis1973ope,B14} and also \cite{ambrosio2005gradient, de1980problems,gigli2023giorgi, sandier2004gamma} for a generalization in metric spaces.
For the specific framework of uniform $\lambda$-convex functionals we refer for instance to \cite{nochetto2000posteriori, ortner2005two, ortner2006gradient, crismale2023variational}.

Here, we recall the stability result as appears in \cite[Theorem 3.8]{crismale2023variational}, that best suits our purposes. We start by recalling the notion of $\lambda$-positive and $\lambda$-convex functions.

Let $\h$ be a Hilbert space endowed with the scalar product $\langle\cdot,\cdot\rangle_\h$ and the norm $|\cdot|_\h$ and let $\lambda>0$.
We say that a functional $\F : \h \rightarrow (-\infty,+\infty] $ is $\lambda$-convex if the functional $\F(\cdot)+ \frac{\lambda}{2} \vert \cdot \vert^{2}_{\h}$ is convex.
Moreover, we say that $\F$ is $\lambda$-positive if $\F(x)+ \frac{\lambda}{2} \vert x \vert^{2}_{\h} \geq 0$ for every $x \in \h$.

In the sequel, we will write $\D(\F)= \{ x\in\h \colon \F(x)\neq + \infty\}$.
For every $x\in\D(\F)$, $\partial\F(x)$ will denote the \emph{subdifferential} of $\F$ at $x$, namely the set
$$
\partial\F(x) := \Big\{v\in\h \colon \liminf_{y\to x} \frac{\F(y)-\F(x)-\langle v,y-x\rangle_\h}{|y-x|_\h }\ge0 \Big\}.
$$

\begin{theorem}\label{genstab}
Let $\{\F^n\}_{n\in\N}$ with $\F^n:\h\to(-\infty,+\infty]$ for every $n\in\N$ be a sequence of proper, strongly lower semicontinuous functionals which are $\lambda$-convex and $\lambda$-positive, for some $\lambda>0$ independent of $n$.
Let $\{x_0^n\}_{n\in\N}\subset\h$ be such that $x_0^n\in\D(\F^n)$ for every $n\in\N$,
$$
\sup_{n\in\N}\F^n(x_0^n)<+\infty,
$$
and $x_0^n\to x^\infty_0$ for some $x^\infty_0\in\h$.
Assume that one of the following statements is satisfied:
		\begin{itemize}
			\item[(a)] The functionals $\F^n$ are (i.e., if the sublevels of the functions $\F^n$ are equibounded) and Mosco-converge to some proper functional $\F^\infty$ with respect to the weak/strong $\h$-convergence (as $n\to +\infty$)\,. 
			\item[(b)] The functionals $\F^n$ $\Gamma$-converge to some proper functional $\F^\infty$ with respect to the strong $\h$-convergence (as $n\to +\infty$) and every sequence $\{y^n\}_{n\in\N}\subset\h$ with
$$
\sup_{n\in\N}\F^n(y^n)+\frac{\lambda}{2}|y^n|_\h^2<+\infty,
$$
admits a strongly convergent subsequence.
		\end{itemize}
		Then, $x_0^\infty\in\D(\F^{\infty})$ and, for every $T>0$, there exists a unique solution $x^n$ to the problem
		\begin{equation}\label{caun}
			\begin{cases}
				\dot{x}(t)\in -\partial \F^n(x(t))  \qquad \text{for a.e.\ }t \in (0,T)\,, \\
				x(0)=x_0^n
			\end{cases}
		\end{equation}
		for every $n\in\N$, there exists a unique solution $x^\infty$ to the problem
		\begin{equation}\label{cauninf}
			\begin{cases}
				\dot{x}(t)\in -\partial \F^\infty(x(t)) \qquad \text{for a.e.\ }t \in (0,T)\,, \\
				x(0)=x^\infty_0\,
			\end{cases}
		\end{equation}   
		and $x^n$ weakly converge to $x^\infty$, as $n\to+\infty$, in $H^1([0,T];\h)$.
		Furthermore, if 
		\begin{equation}\label{reco0}
			\lim_{n\to +\infty} \F^n(x^n_0)=\F^\infty(x^\infty_0)\,,
		\end{equation} 
		then,  we have that 
		\begin{equation}\label{strongc}
			x^n\to x^\infty\qquad\textrm{(strongly) in }H^1([0,T];\h)\quad\textrm{ for every }T>0\,,
		\end{equation}
		\begin{equation}\label{alwaysrs}
			x^n(t)\overset{\h}{\to}x^{\infty}(t)\quad\textrm{and}\quad\F^n(x^n(t))\to\F^{\infty}(x^{\infty}(t))\qquad\textrm{for every }t\ge 0\,. 
		\end{equation}
	\end{theorem}

We now state the following general result which connects the notion of first variation with that of \emph{subgradient} (\emph{i.e.}\ an element of the subdifferential).
For a proof of the result below we refer to \cite[Proposition 3.7]{crismale2023variational}.
For every vector space $\mathscr{V}$ we denote by $\mathscr{V}^*$ the algebraic dual space of $\mathscr{V}$ and by $\mathscr{V}'$ the topological dual space of $\mathscr{V}$. 

	\begin{proposition}\label{singleton}
	Let $\F:\h\to(-\infty,+\infty]$ be a proper lower semicontinuous functional which is $\lambda$-convex, for some $\lambda>0$ and let $x\in\D(\F)$.
	Let $\hat\h$ be a dense subspace of $\h$.
	If there exists $T\in(\hat\h)^*$ such that 
		\begin{equation}\label{grad}
		\lim_{t\to 0}\frac{\F(x+ty)-\F(x)}{t}=T(y)\qquad\textrm{for every }y\in\hat\h\,,
		\end{equation}
	then, either $\partial\F(x)=\emptyset$ or $\partial\F(x)=\{v\}$, where $v$ is the (unique) element in $\h$  satisfying $T(y)=\langle v,y\rangle_\h$ for every $y\in\hat\h$.
	In particular, $T\in (\hat\h)'$ and $v$ is its  unique continuous extension to $\h'$.
	\end{proposition}

\subsection{Convergence of the $\G_s^1$ gradient flows to the gradient flows of $\G^1$}
Our strategy to prove the convergence of the gradient flows of the $\G_s^1$ energies is to prove that the family $\{\G_s^1\}_{s\in(\frac{1}{2},1)}$ satisfies the hypotheses of Theorem \ref{genstab}.
To this end, we only need to prove the uniform $\lambda$-positivity and $\lambda$-convexity as the hypothesis that energy $\G_s^1$ converge in sense of Mosco to $\G^1$ is already proved in the first part of the article.

	\begin{lemma}\label{lambdapositiveconverxGs}
	For every $\lambda> 8 N\omega_N$, the functionals $ \G_s^1$ are $\lambda$-positive and $\lambda$-convex for every $s \in \left(\frac{1}{2},1\right)$.
	\end{lemma}
	\begin{proof}
	We start by noticing that, by decomposition \eqref{eq:Gs-dec} and by the positivity result provided by Lemma \ref{monotonicitylemma}, in order to prove the $\lambda$-positivity of $\G^1_s$ it is sufficient to prove that
	$$
	\Big(\frac{\lambda}{2}-\frac{N\omega_N}{s}\Big)\int_{\R^N}|u(x)|^2 \de x+ 2\int_{\R^N}\int_{B(0,1)^c}\frac{u(x+h)u(x)}{|h|^{N+2s}} \de h \de x\ge 0.
	$$
	This, in turn, is an immediate consequence of \eqref{01112023matt1} and the fact that $s\ge\frac{1}{2}$.

	Now we show that the functionals $\G_s^1$ are $\lambda$-convex for every $s \in \left(\frac{1}{2},1\right)$.
	For every $u,v\in L^2(\Omega)$ we define the function $f:\R\to\R$ as
	$$
	f(t):=\Big(\frac{\lambda}{2}-\frac{N\omega_N}{s}\Big)\int_{\R^N}|u(x)+t v(x)|^2 \de x+ 2\int_{\R^N}\int_{B(0,1)^c}\frac{(u(x+h)+t v(x+h))(u(x)+t v(x))}{|h|^{N+2s}} \de h \de x
	$$
	and claim that $\frac{d^2}{d t^2}f(t) \geq0$ for every $t \in \R$. 
	Indeed, again by \eqref{01112023matt1} we obtain
	\begin{align*}
	\frac{d^2}{d t^2}f(t) &= \Big(\lambda-\frac{2N\omega_N}{s}\Big)\int_{\R^N}|v(x)|^2\de x+4\int_{\R^N}\int_{B(0,1)^c}\frac{v(x+h)v(x)}{|h|^{N+2s}} \de h \de x \\
	& \ge \big(\lambda-8N \omega_N\big) \int_{\R^N}|v(x)|^2\de x \ge 0.
	\end{align*}
	With analogous computations, defining $g:\R\to\R$ as $g(t):=\mathcal{J}_s(u+tv)$ we obtain, as a consequence again of Lemma \ref{monotonicitylemma},
	$$
	\frac{d^2}{d t^2}g(t) = 2 \mathcal{J}_s(v,v)\geq 0.
	$$
	Thanks to decomposition \eqref{eq:Gs-dec}, the two inequalities above then yield the result.
	\end{proof}

	We have now everything in place to prove the convergence of the gradient flows corresponding to the $\G_s^1$ functionals as $s \rightarrow 1$.
	The strategy of the proof follows \cite[Theorems 4.4--4.6]{crismale2023variational} and \cite[Theorems 4.2--4.6]{de2023parabolic}.

	\begin{theorem}\label{convheat1ordsto0}
	Let $\{s_n\}_{n\in\N}\subset (0,1)$ be such that $s_n\to 1^-$ as $n\to +\infty$.
	Let $u^\infty_0\in L^2(\Omega)$ and let $\{u^{n}_0\}_{n\in\N}\subset H_0^1(\Omega)$ be such that $\sup_{n\in\N}\G^1_{s_n}(u^n_0)<+\infty$ and $u^n_0\to u^\infty_0$ in $L^2(\Omega)$.
	Then, for every $n\in\N$ there exists a unique solution $u^n\in H^1_{loc}([0,+\infty); L^2(\Omega))\cap L^\infty([0,+\infty);H^1_0(\Omega))$ to
	\begin{equation}\label{cauchyord1n}
		\begin{dcases}
		u_t(t) = -\frac{\omega_N(-\Delta)-\frac{4(1-s)}{C(N,s)}(-\Delta)^s}{1-s} u(t) \qquad\textrm{for a.e.\ }t\in (0,+\infty)\,, \\
		u(0)=u^{n}_0\,,
		\end{dcases}
	\end{equation}
	such that  $u^n(t)\in H^1_0(\Omega)$ for  $a.e.\, t\ge 0$.
	Moreover, $u^\infty_0\in\mathfrak{D}(\G^1)$ and, for every $T>0$,  $u^n\to u^\infty$ in the weak topology of $H^1([0,T];L^2(\Omega))$ as $n\to +\infty$\,, where $u^\infty\in H^1_{loc}([0,+\infty);L^2(\Omega))\cap L^\infty([0,+\infty);H^1_0(\Omega))$ is the unique solution to
	\begin{equation}\label{cauchyord1infty}
	\begin{dcases}
		u_t(t) =-\mathfrak{L}  u(t)\qquad\textrm{for a.e.\ }t\in (0,+\infty) \\
		u(0)=u^\infty_0\,,
	\end{dcases}
	\end{equation}
	and such that  $ u^\infty(t)\in \mathfrak{D}(\G^1)$ for $a.e. \,t\ge 0$\,.
	Furthermore, if
		$$
		\lim_{n\to +\infty} \G_{s_n}^1(u^n_0)= \G^1(u^\infty_0)\,,
		$$
		then,  $u^n\to u^\infty$ (strongly) in $H^1([0,T];L^2(\Omega))$ for  every  $T>0$\,, and
		$$
		\|u^n(t)-u^\infty(t)\|_{L^2(\Omega)}\to 0\quad\textrm{and}\quad \G_{s_n}^1(u^n(t))\to \G^1(u^\infty(t))\qquad\textrm{for every }t\ge 0\,.
		$$
	\end{theorem}
	\begin{proof}
	By the very definition of $ \G_s^1$, see \eqref{G_s^1}, the domain of $\G_{s_n}^1$ is $H_0^1(\Omega)$ for all $n \in \N$ and from Proposition \ref{proplsc+com+sfix} $\G^1_{s_n}$ is lower semicontinuous with respect to the strong $L^2$ topology.
	By Lemma \ref{lambdapositiveconverxGs} $\G^1_{s_n}$ is $\lambda$-positive and $\lambda$-convex for every $\lambda > 4N \omega_N+4 \vert \Omega \vert$; here we can assume for simplicity and without loss of generality that $s_n\ge\frac{1}{2}$ for every $n\in\N$.

	Now by formula \eqref{06112023matt1} and Proposition \ref{singleton} applied with $ \mathcal{F}= \G^1_{s_n}$, $ \h= L^2(\Omega)$ and $ \hat{\h}= C_c^{\infty}(\Omega)$ we have that for every $ u \in H^1_0(\Omega)$, either $ \partial \G_{s_n}^1(u)= \emptyset$ or
	$$
	\partial \G_{s_n}^1(u)=\Big\{\frac{\omega_N(-\Delta)-\frac{4(1-s)}{C(N,s)}(-\Delta)^s}{1-s}u\Big\},
	\quad \text{with} \quad
	\frac{\omega_N(-\Delta)-\frac{4(1-s)}{C(N,s)}(-\Delta)^s}{1-s} u \in L^2(\Omega).
	$$
	Furthermore, by Lemma \ref{varprimaG^1} and by Proposition \ref{singleton} applied with $ \mathcal{F}= \G^1$, $\h=L^2(\Omega)$ and $ \hat{\h}= C_c^{\infty}(\Omega)$ we have that every $ u \in \mathfrak{D}(\G^1)$, either $ \partial \G^1(u)= \emptyset $ or $\partial \G^1(u)= \{\mathfrak{L}  u\}$ with $ \mathfrak{L}  u \in L^2(\Omega)$.

	By Theorem \ref{genstab} the solutions to problems \eqref{cauchyord1n} and \eqref{cauchyord1infty} exist and are unique and belong to $L^2(\Omega)$ for $a.e.$ $t\ge0$.
	The fact that $\sup_{t\in[0,+\infty)}\|u^n(t)\|_{H^1}<+\infty$ is a direct consequence of Proposition \ref{compcompcompthm} and the fact that
	$$
	\sup_n \G^1_{s_n}(u^n(t))\le \sup_n \G^1_{s_n}(u_0^n)<+\infty.
	$$
	This together with Theorem \ref{thm:main} (i) yields also that $\sup_{t\in[0,+\infty)}\|u^\infty(t)\|_{H^1}<+\infty$.
	Eventually, the stability claim follows by applying Theorem \ref{genstab} with $\h=L^2(\Omega)$, $ \mathcal{F}_n= \G_{s_n}^1$ and $ \mathcal{F}^{\infty}= \G^1$, once noticed that, in view of Theorem \ref{thm:main},both assumptions $(a)$ and $(b)$ of Theorem \ref{genstab} are satisfied.
	\end{proof}

	\subsection*{Acknowledgements}  
		A. K. is supported by the DFG Collaborative Research Center TRR 109 “Discretization in Geometry and Dynamics” and by the Academy of Finland grant 314227. 
	V.P. acknowledges support by the Austrian Science Fund (FWF) through projects
	\href{https://doi.org/10.55776/I4052}{10.55776/I4052} and \href{https://doi.org/10.55776/Y1292}{10.55776/Y1292},
	as well as from OeAD through the WTZ grant CZ 09/2023.
	VP was a member of INdAM-GNAMPA and	acknowledges support from the project CUP\_E53C22001930001.
	A.T.\ acknowledges funding by the Deutsche Forschungsgemeinschaft (DFG, German Research Foundation) through SPP 2256, project ID 441068247.

	
	{\normalsize
	\bibliography{biblio}
	}
\end{document}